\documentclass[10pt]{amsart}
\usepackage{amssymb,mathrsfs}
\usepackage[T1]{fontenc}
\usepackage{lmodern}
\usepackage{color}

\def\bR {\mathbf{R}}
\def\bS {\mathbf{S}}

\def\bZ {\mathbf{Z}}

\def\cA {\mathcal{A}}
\def\cB {\mathcal{B}}
\def\cC {\mathcal{C}}
\def\cD {\mathcal{D}}

\def\cL {\mathcal{L}}
\def\cM {\mathcal{M}}
\def\cN {\mathcal{N}}

\def\cQ {\mathcal{Q}}
\def\cR {\mathcal{R}}
\def\cS {\mathcal{S}}
\def\cT {\mathcal{T}}

\def\a {{\alpha}}
\def\b {{\beta}}
\def\g {{\gamma}}

\def\de {{\delta}}
\def\eps {{\epsilon}}
\def\th {{\theta}}

\def\ka {{\kappa}}
\def\l {{\lambda}}

\def\si {{\sigma}}
\def\Si {{\Sigma}}

\def\om {{\omega}}

\def\d {{\partial}}
\def\grad {{\nabla}}
\def\Dlt {{\Delta}}

\def\la {\langle}
\def\ra {\rangle}
\def \La {\bigg\langle}
\def \Ra {\bigg\rangle}

\def\wto {{\rightharpoonup}}

\def\dd {\mathrm{d}}

\newcommand{\Div}{\operatorname{div}}

\newcommand{\Dom}{\operatorname{Dom}}

\newcommand{\Span}{\operatorname{Span}}

\newcommand{\Tr}{\operatorname{trace}}

\newcommand{\Ker}{\operatorname{Ker}}
\newcommand{\Img}{\operatorname{Im}}

\newcommand{\ba}{\begin{aligned}}
\newcommand{\ea}{\end{aligned}}

\newcommand{\be}{\begin{equation}}
\newcommand{\ee}{\end{equation}}
\newcommand{\lb}{\label}

\newtheorem{theorem}{Theorem}[section]

\newtheorem{lemma}[theorem]{Lemma}
\newtheorem{proposition}{Proposition}

\theoremstyle{definition}

\definecolor{violetto}{rgb}{0.8, 0.4, 0.8 }

\begin{document}

\title[Vlasov-Navier-Stokes Model for Aerosol Flows]{A Derivation of the Vlasov-Navier-Stokes Model for Aerosol Flows from Kinetic Theory}

\author[E. Bernard]{Etienne Bernard}
\address[E.B.]{IGN-LAREG, Universit\'e Paris Diderot, B\^atiment Lamarck A, 5 rue Thomas Mann, Case courrier 7071, 75205 Paris Cedex 13, France}
\email{esteve.bernard@gmail.com}

\author[L. Desvillettes]{Laurent Desvillettes}
\address[L.D.]{Universit\'e Paris Diderot, Sorbonne Paris Cit\'e, Institut de Math\'ematiques de Jussieu - Paris Rive Gauche, UMR CNRS 7586, 75013, Paris, France}
\email{desvillettes@math.univ-paris-diderot.fr}

\author[F. Golse]{Fran\c cois Golse}
\address[F.G.]{CMLS, Ecole polytechnique et CNRS, Universit\'e Paris-Saclay, 91128 Palaiseau Cedex, France}
\email{francois.golse@polytechnique.edu}

\author[V. Ricci]{Valeria Ricci}
\address[V.R.]{Dipartimento di Matematica e Informatica, Universit\`a degli Studi di Palermo, Via Archirafi 34, 90123 Palermo, Italy}
\email{valeria.ricci@unipa.it}

\begin{abstract}
This article proposes a derivation of the Vlasov-Navier-Stokes system for spray/aerosol flows. The distribution function of the dispersed phase is governed by a Vlasov-equation, while the velocity field of the propellant satisfies the Navier-Stokes
equations for incompressible fluids. The dynamics of the dispersed phase and of the propellant are coupled through the drag force exerted by the propellant on the dispersed phase. We present a formal derivation of this model from a multiphase 
Boltzmann system for a binary gaseous mixture, involving the droplets/dust particles in the dispersed phase as one species, and the gas molecules as the other species. Under suitable assumptions on the collision kernels, we prove that the 
sequences of solutions to the multiphase Boltzmann system converge to distributional solutions to the Vlasov-Navier-Stokes equation in some appropriate distinguished scaling limit. Specifically, we assume (a) that the mass ratio of the gas 
molecules to the dust particles/droplets is small, (b) that the thermal speed of the dust particles/droplets is much smaller than that of the gas molecules and (c) that the mass density of the gas and of the dispersed phase are of the same order of
magnitude. The class of kernels modelling the interaction between the dispersed phase and the gas includes, among others, elastic collisions and inelastic collisions of the type introduced in [F. Charles: in ``Proceedings of the 26th International 
Symposium on Rarefied Gas Dynamics'', AIP Conf. Proc. 1084, (2008), 409--414].
\end{abstract}



\maketitle


\section{Introduction}


An aerosol (or a spray) is a complex fluid consisting of a \textit{dispersed phase}, for instance solid particles or liquid droplets, immersed in a gas, sometimes referred to as the \textit{propellant}.

An important class of models for the dynamics of aerosol/spray flows consists of

(a) a kinetic equation for the dispersed phase, and

(b) a fluid equation for the background gas.

The kinetic equation for the dispersed phase and the fluid equation for the back\-ground gas are coupled through the drag force exerted by the gas on the drop\-lets/par\-ti\-cles.

A well-known example of this class of models is the (incompressible) Vlasov-Navier-Stokes system:
$$
\left\{
\ba
{}&\d_tF+v\cdot\grad_xF-\frac{\ka}{m_p}\Div_v((v-u)F)=0\,,
\\
&\rho_g(\d_tu+u\cdot\grad_xu)+\grad_xp=\rho_g\nu\Dlt_xu+\ka\int_{\bR^3}(v-u)F\,\dd v\,,
\\
&\Div_xu=0\,.
\ea
\right.
$$
The unknowns in this system are $F\equiv F(t,x,v)\ge 0$, the distribution function of the dispersed phase, i.e. the number density of particles or droplets with velocity $v$ located at the position $x$ at time $t$, and $u\equiv u(t,x)\in\bR^3$,
the velocity field in the gas. The parameters $\ka$, $m_p$, $\rho_g$ and $\nu$ are positive constants. Specifically, $\ka$ is the friction coefficient of the gas on the dispersed phase, $m_p$ is the mass of a particle or droplet, and $\rho_g$ is 
the gas density, while $\nu$ is the kinematic viscosity of the gas. The last equation in the system above indicates that the gas flow is considered as incompressible\footnote{It is well known that the motion of a gas at a very low Mach number 
is governed by the equations of incompressible fluid mechanics, even though a gas is a compressible fluid. A formal justification for this fact can be found on pp. 11--12 in \cite{PLLFluidMech1}.}. The scalar pressure field $p\equiv p(t,x)\in\bR$ 
is instantaneously coupled to the unknowns $F$ and $u$ by the Poisson equation
$$
-\Dlt_xp=\rho_g\Tr((\grad_xu)^2)-\ka\Div_x\int_{\bR^3}(v-u)F\,\dd v\,.
$$
The mathematical theory of the Vlasov-Navier-Stokes system has been discussed in \cite{ChengYu}. Various asymptotic limits of the Vlasov-Stokes system that are of great interest in the modeling of aerosol or spray flows have been
investigated in the mathematical literature: see for instance \cite{GoudonJabinVasseur1, GoudonJabinVasseur2, LDM}.

Our aim in the present work is different: we are concerned in deriving models such as the Vlasov-Navier-Stokes system from a more microscopic description of aerosol flows.

Perhaps the most natural idea for doing so is to view the Vlasov-Navier-Stokes system as a mean field model governing the limit of the solid particles (or droplets) phase space empirical measure as the particle number tends to infinity 
and the particle radius vanishes in some appropriate distinguished scaling. 

Derivations of the Stokes and Navier-Stokes equation with {a force term including} the drag force exerted by the particles on the fluid (known as the Brinkman force) from a system consisting of a large number of particles immersed in a 
viscous fluid can be found in \cite{Allaire, DesvFGRicci08}. While these papers obtain the same Navier-Stokes equation as in the coupled system above (more precisely, its steady variant), they assume that the phase space distribution 
of particles or droplets is given, and therefore do not derive the full Vlasov-Navier-Stokes system. The reason for this shortcoming is the following: both references  \cite{Allaire, DesvFGRicci08} use the method of homogenization of elliptic 
operators with holes of finite capacity pioneered by Khruslov and his school --- see for instance \cite{Khruslov,CioraMurat}. Unfortunately, these methods assume that the minimal distance between particles remains uniformly much larger 
than the particle radius $r\ll 1$ --- specifically, of the order of $r^{1/3}$ in space dimension $3$. While this assumption can be imposed if the distribution of particles is given, such a control on the distance between neighboring particles is 
probably not nicely propagated by the particle dynamics and most likely hard to establish (see however \cite{JabinOtto} for interesting ideas in this direction). Even if such a control could be established, configurations of $N$ particles with 
such a uniform control on the minimal distance between neighboring particles are of vanishing probability in the large $N$ limit: see for instance Proposition 4 in \cite{Hauray}. Worse, the coupled dynamics of finitely many rigid spheres 
immersed in a Navier-Stokes flow may not be defined for all positive times: see \cite{DesjaEsteban, DGVHill}.

In view of all these difficulties, we have chosen another route to derive coupled systems such as the Vlasov-Navier-Stokes system from a more microscopic model. Specifically, we start from a coupled system of Boltzmann equations for the 
solid particles or droplets and for the gas molecules.

One might object that the Boltzmann equation is a first principle equation neither for the solid particles nor for the gas molecules. In addition, the idea to treat the gas molecules and the particles in the dispersed phase on equal footing is 
most unnatural. On the other hand, the system of Newton's equations written for each solid particle immersed in an incompressible Navier-Stokes fluid cannot be considered as a first principle model for aerosol flows either. Indeed, it is 
only in some very special asymptotic limit that the dynamics of a gas is governed by the incompressible Navier-Stokes equations.

On the other hand, using a system of Boltzmann equations for the dispersed phase and the gas allows considering distributions of solid particles or droplets without any constraint on the minimal distance between neighbouring particles. 
In fact collisions between particles in the dispersed phase are described by a collision integral, in the same way as collisions between gas molecules. 

Another benefit in this approach is the great variety of models describing the interaction between the dispersed phase and the propellant. In the present work, this interaction is described in terms of a general class of Boltzmann type collision 
integrals, assumed to satisfy a few assumptions discussed in section \ref{S-3} below. We have focussed our attention on two examples of such collision integrals; in one case, collision are assumed to be elastic, while the other example is
based on the diffuse reflection of gas molecules on the surface of dust particles or droplets, which is an inelastic process. 

For that reason, we believe that the idea of starting from the kinetic theory of multicomponent gases may provide an interesting alternative to the traditional arguments used in deriving the various dynamical models appearing in the theory 
of aerosol flows. 

This approach should not be confused with the more detailed analysis of rarefied gas flows past an immersed body (see for instance \cite{Takata,TaguchiJFM,TaguchiRGD}). It has been known for a long time that the motion of an immersed
body in a viscous fluid involves nonlocal effects in the time variable: see \cite{Boussinesq85} for a short, yet detailed presentation of the Boussinesq-Basset force. Similar effects can be observed in the case of a solid particle immersed in
a rarefied gas and have been recently studied: see \cite{AokiPulvi} and the references therein. Since our description of the interaction between the dispersed phase and the propellant is based on collision integrals, it does not include such
effects. On the other hand, our purpose is not to focus on the details of the interaction between a single dust particle or droplet with the propellant, but rather to investigate the collective behavior of the dispersed phase. Whether the system
of Boltzmann equations for a 2-component gas can be justified from a more detailed, microscopic model, such as the dynamics of a system of solid particles immersed in a rarefied gas, seems to be a very interesting problem, albeit a very
difficult one.

Our derivation of dynamical equations for aerosol flows from the kinetic theory of multicomponent gas is systematic yet formal, in the sense of the derivations of fluid dynamic equations from the Boltzmann equation in \cite{BGL1}.
The outline of this paper is the following: section \ref{S-2} introduces the system of Boltzmann equations used as the starting point in our derivation. In particular, the fundamental conservation laws and basic properties of this system
are recalled in section \ref{S-2}, along with the dimensionless form of the equations and the definition of the scaling parameters involved. Section \ref{S-3} studies in detail the main properties of the collision kernel describing the interaction
of gas molecules with dust particles or droplets. Our main result, i.e. is the derivation of the Vlasov-Navier-Stokes system, is stated as Theorem \ref{T-LimThm} in section \ref{S-4}. Its proof occupies most of section \ref{S-4}.


\section{Boltzmann Equations for Multicomponent Gases}\lb{S-2}


Consider a binary mixture consisting of microscopic gas molecules and much bigger solid dust particles or liquid droplets. For simplicity, we henceforth assume that the dust particles or droplets are identical (in particular, the spray is 
monodisperse: all particles have the same mass), and that the gas is monatomic. We denote from now on by $F\equiv F(t,x,v)\ge 0$ the distribution function of dust particles or droplets, and by $f\equiv f(t,x,w)\ge 0$ the distribution 
function of gas molecules. These distribution functions satisfy the system of Boltzmann equations
\be\lb{BoltzSys0}
\ba
(\d_t+v\cdot\grad_x)F&=\cD(F,f)+\cB(F)\,,
\\
(\d_t+w\cdot\grad_x)f&=\cR(f,F)+\cC(f)\,.
\ea
\ee
The terms $\cB(F)$ and $\cC(f)$ are the Boltzmann collision integrals for pairs of dust particles or liquid droplets and for pairs of gas molecules respectively. The terms $\cD(F,f)$ and $\cR(f,F)$ are Boltzmann type collision integrals 
describing the deflection of dust particles or liquid droplets subject to the impingement of gas molecules, and the slowing down of gas molecules by collisions with dust particles or liquid droplets respectively.

\subsection{Fundamental conservation laws for multicomponent Boltzmann systems}


Before describing in detail the collision integrals introduced above, we recall their fundamental properties.

Collisions between gas molecules are assumed to be elastic, so that the Boltzmann collision integral $\cC(f)$ satisfies the following local conservation laws of mass, momentum and energy: for each measurable $f$ defined a.e. on $\bR^3$ 
and rapidly decaying as $|w|\to\infty$,
\be\label{star}
\int_{\bR^3}\cC(f)(w)\left(\begin{matrix} 1\\ w\\ |w|^2\end{matrix}\right)\,\dd w=0\,.
\ee

Collisions between dust particles or liquid droplets may not be perfectly elastic, so that the Boltzmann collision integral $\cB(F)$ satisfies only the local conservation laws of mass and momentum: for each measurable $F$ defined a.e. on 
$\bR^3$ and rapidly decaying as $|v|\to\infty$,
\be\label{deuxpr}
\int_{\bR^3}\cB(F)(v)\left(\begin{matrix} 1\\ v\end{matrix}\right)\,\dd v=0\,.
\ee

Collisions between gas molecules and dust particles or liquid droplets obviously preserve the nature of the colliding objects. Therefore, the collision integrals $\cD(F,f)$ and $\cR(f,F)$ satisfy the following local conservation of particle number 
per species: for each measurable $F$ and $f$ defined a.e. on $\bR^3$ and rapidly decaying at infinity, 
\be\label{starstar}
\int_{\bR^3}\cD(F,f)(v)\,\dd v=\int_{\bR^3}\cR(f,F)(w)\,\dd w=0\,.
\ee
These collision integrals satisfy the local balance of momentum in the aerosol, i.e.
\be\label{troispr}
m_p\int_{\bR^3}\cD(F,f)(v)v\,\dd v+m_g\int_{\bR^3}\cR(f,F)(w)w\,\dd w=0\,,
\ee
where $m_g$ is the mass of gas molecules and $m_p$ the mass of dust particles or liquid droplets. 

If the collisions between gas molecules and droplets or dust particles are elastic, these collision integrals satisfy in addition the local balance of energy in the aerosol, i.e.
$$
m_p\int_{\bR^3}\cD(F,f)(v)\tfrac12|v|^2\,\dd v+m_g\int_{\bR^3}\cR(f,F)(w)\tfrac12|w|^2\,\dd w=0\,.
$$

\subsection{Dimensionless Boltzmann systems}


We assume for simplicity that the aerosol is enclosed in a periodic box of size $L>0$, i.e. $x\in\bR^3/L\bZ^3$. The system of Boltzmann equations (\ref{BoltzSys0}) involves an important number of physical parameters, which are listed in the 
table below.

\bigskip
\begin{center}
\begin{tabular}{|c|c|}
\hline
\hspace{.2cm} Parameter \hspace{.2cm} & \hspace{.2cm} Definition \hspace{.2cm}\\
\hline
\hline
\hspace{.2cm} $L$ \hspace{.2cm} & \hspace{.2cm} size of the container (periodic box) \hspace{.2cm}\\
\hline
\hspace{.2cm} $\cN_p$ \hspace{.2cm} & \hspace{.2cm} number of particles$/L^3$ \hspace{.2cm}\\
\hline
\hspace{.2cm} $\cN_g$ \hspace{.2cm} & \hspace{.2cm} number of gas molecules$/L^3$ \hspace{.2cm}\\
\hline
\hspace{.2cm} $V_p$ \hspace{.2cm} & \hspace{.2cm} thermal speed of particles \hspace{.2cm}\\
\hline
\hspace{.2cm} $V_g$ \hspace{.2cm} & \hspace{.2cm} thermal speed of gas molecules \hspace{.2cm}\\
\hline
\hspace{.2cm} $S_{pp}$ \hspace{.2cm} & \hspace{.2cm} average particle/particle cross-section \hspace{.2cm}\\
\hline
\hspace{.2cm} $S_{pg}$ \hspace{.2cm} & \hspace{.2cm} average particle/gas cross-section \hspace{.2cm}\\
\hline
\hspace{.2cm} $S_{gg}$ \hspace{.2cm} & \hspace{.2cm} average molecular cross-section \hspace{.2cm}\\
\hline
\hspace{.2cm} $\eta=m_g/m_p$ \hspace{.2cm} & \hspace{.2cm} mass ratio (molecules/particles) \hspace{.2cm}\\
\hline
\hspace{.2cm} $\eps=V_p/V_g$ \hspace{.2cm} & \hspace{.2cm} thermal speed ratio (particles/molecules) \hspace{.2cm}\\
\hline
\end{tabular}

\bigskip
\centerline{Table 1: the physical parameters for binary gas mixtures}
\end{center}

\bigskip
We first define a dimensionless position variable:
$$
\hat x:=x/L\,,
$$
together with dimensionless velocity variables for each species:
$$
\hat v:=v/V_p\,,\quad \hat w:=w/V_g\,.
$$
In other words, the velocity of each species is measured in terms of the thermal speed of the particles in the species under consideration. 

Next, we define a time variable, which is adapted to the individual motion of the typical particle of the slowest species, i.e. the dust particles or droplets:
$$
\hat t:=tV_p/L\,.
$$

Finally, we define dimensionless distribution functions for each particle species:
$$
\hat F(\hat t,\hat x,\hat v):=V^3_pF(t,x,v)/\cN_p\,,\qquad\hat f(\hat t,\hat x,\hat w):=V^3_gf(t,x,w)/\cN_g\,.
$$

The definition of dimensionless collision integrals is more complex and involves the average collision cross-sections $S_{pp},S_{pg},S_{gg}$, whose definition is recalled below. 

The collision integrals $\cB(F)$, $\cC(f)$, $\cD(F,f)$ and $\cR(f,F)$ are given by expressions of the form
\be
\label{Colli}
\ba
\cB(F)(v)=&\iint_{\bR^3\times\bR^3}F(v')F(v'_*)\Pi_{pp}(v,\dd v'\,\dd v'_*)
\\
&-F(v)\int_{\bR^3}F(v_*)|v-v_*|\Si_{pp}(|v-v_*|)\,\dd v_*\,,
\\
\cC(f)(w)=&\iint_{\bR^3\times\bR^3}f(w')f(w'_*)\Pi_{gg}(w,\dd w'\,\dd w'_*)
\\
&-f(w)\int_{\bR^3}f(w_*)|w-w_*|\Si_{gg}(|w-w_*|)\,\dd w_*\,,
\\
\cD(F,f)(v)=&\iint_{\bR^3\times\bR^3}F(v')f(w')\Pi_{pg}(v,\dd v'\,\dd w')
\\
&-F(v)\int_{\bR^3}f(w)|v-w|\Si_{pg}(|v-w|)\,\dd w\,,
\\
\cR(f,F)(w)=&\iint_{\bR^3\times\bR^3}F(v')f(w')\Pi_{gp}(w,\dd v'\,\dd w')
\\
&-f(w)\int_{\bR^3}F(v)|v-w|\Si_{pg}(|v-w|)\,\dd v\,.
\ea
\ee
In these expressions, $\Pi_{pp},\Pi_{gg},\Pi_{pg},\Pi_{gp}$ are nonnegative, measure-valued measurable functions defined a.e. on $\bR^3$, while $\Si_{pp},\Si_{gg},\Si_{pg}$ are nonnegative measurable functions defined a.e. on $\bR_+$. 

The quantities $\Pi$ and $\Si$ are related by the following identities:
\be\label{Colli2}
\ba 
\int_{\bR^3_v}\dd v\,\Pi_{pp}(v,\dd v'\,\dd v'_*) &=  |v'-v'_*|\Si_{pp}(|v'-v'_*|)\,\dd v'\,\dd v'_*\,,
\\
\int_{\bR^3_w} \dd w\, \Pi_{gg}(w,\dd w'\,\dd w'_*) &= |w'-w'_*|\Si_{gg}(|w'-w'_*|)\,\dd w'\,\dd w'_*\,,
\\
\int_{\bR^3_v} \dd v\,\Pi_{pg}(v,\dd v'\,\dd w') &= |v'-w'|\Si_{pg}(|v'-w'|) \,\dd v'\,\dd w'\,,
\\
\int_{\bR^3_w}  \, \dd w\,\Pi_{gp}(w,\dd v'\,\dd w') &= |v'-w'|\Si_{pg}(|v'-w'|) \,\dd v'\,\dd w'\,.
\ea
\ee
In each one of these identities, the left hand side is to be understood as an integral with respect to the unprimed variable ($v$ for the 1st and 3rd identities, $w$ for the 2nd and the 4th) of a nonnegative measurable function with values in
the set of Borel measures on $\bR^3\times\bR^3$. The identities (\ref{Colli2}) imply the conservation of mass for each species of particles for all the collision integrals appearing in  (\ref{Colli}). These conservation laws have been stated
above: see the first lines of (\ref{star}) and (\ref{deuxpr}), and (\ref{starstar}).

We refer to formula (3.6) in \cite{Landau10} for this general presentation of collision integrals. Specific examples of these collision integrals will be discussed in section \ref{SS-ExplctCollInt} below

According to formula (2.2) in \cite{Landau10}, $\Si_{pp},\Si_{gg}$ and $\Si_{pg}$ have the dimensions of areas. The corresponding dimensionless quantities are 
$$
\ba
\hat\Si_{pp}(|\hat v|)&=\Si_{pp}(V_p|\hat v|)/S_{pp}\,,
\\
\hat\Si_{gg}(|\hat w|)&=\Si_{gg}(V_g|\hat w|)/S_{gg}\,,
\\
\hat\Si_{pg}(|\hat z|)&=\Si_{pg}(V_g|\hat z|)/S_{pg}\,.
\ea
$$
Likewise, we define the dimensionless measure-valued collision kernels by the formulas
$$
\ba
\hat\Pi_{pp}(\hat v,\dd\hat v'\,\dd\hat v'_*)&=\Pi_{pp}(v,\dd v'\,\dd v'_*)/S_{pp}V_p^4\,,
\\
\hat\Pi_{gg}(\hat w,\dd\hat w'\,\dd\hat w'_*)&=\Pi_{gg}(w,\dd w'\,\dd w'_*)/S_{gg}V_g^4\,,
\\
\hat\Pi_{pg}(\hat v,\dd\hat v'\,\dd\hat w')&=\Pi_{pg}(v,\dd v'\,\dd w')/S_{pg}V_g^4\,,
\\
\hat\Pi_{gp}(\hat w,\dd\hat v'\,\dd\hat w')&=\Pi_{gp}(w,\dd v'\,\dd w')/S_{pg}V_gV_p^3\,.
\ea
$$

We henceforth define the dimensionless collision integrals as follows:
$$
\ba
\hat\cB(\hat F)(\hat v)=&\iint_{\bR^3\times\bR^3}\hat F(\hat v')\hat F(\hat v'_*)\hat \Pi_{pp}(\hat v,\dd \hat v'\,\dd \hat v'_*)
\\ \label{dimc}
&-\hat F(\hat v)\int_{\bR^3}\hat F(\hat v_*)|\hat v-\hat v_*|\hat \Si_{pp}(|\hat v-\hat v_*|)\,\dd \hat v_*\,,
\\ \label{dimd}
\hat \cC(\hat f)(\hat w)=&\iint_{\bR^3\times\bR^3}\hat f(\hat w')\hat f(\hat w'_*)\hat \Pi_{gg}(\hat w,\dd \hat w'\,\dd \hat w'_*)
\\ \label{dimr}
&-\hat f(\hat w)\int_{\bR^3}\hat f(\hat w_*)|\hat w-\hat w_*|\hat \Si_{gg}(|\hat w-\hat w_*|)\,\dd \hat w_*\,,
\ea
$$
while
$$
\ba
\hat \cD(\hat F,\hat f)(\hat v)=&\iint_{\bR^3\times\bR^3}\hat F(\hat v')f(\hat w')\hat \Pi_{pg}(\hat v,\dd \hat v'\,\dd \hat w')
\\
&-\hat F(\hat v)\int_{\bR^3}\hat f(\hat w)\left|\tfrac{V_p}{V_g}\hat v-\hat w\right|\hat \Si_{pg}\left(\left|\tfrac{V_p}{V_g}\hat v-\hat w\right|\right)\,\dd \hat w\,,
\\
\hat \cR(\hat f,\hat F)(\hat w)=&\iint_{\bR^3\times\bR^3}\hat F(\hat v')\hat f(\hat w')\hat \Pi_{gp}(\hat w,\dd \hat v'\,\dd \hat w')
\\
&-\hat f(\hat w)\int_{\bR^3}\hat F(\hat v)\left|\tfrac{V_p}{V_g}\hat v-\hat w\right|\hat \Si_{pg}\left(\left|\tfrac{V_p}{V_g}\hat v-\hat w\right|\right)\,\dd \hat v\,.
\ea
$$

With the dimensionless quantities so defined, we arrive at the following dimensionless form of the multicomponent Boltzmann system:
\be\lb{BoltzSys}
\left\{
\ba
{}&\d_{\hat t}\hat F\,+\,\hat v\cdot\grad_{\hat x}\hat F\,=\cN_gS_{pg}L\frac{V_g}{V_p}\hat\cD(\hat F,\hat f)+\cN_pS_{pp}L\hat\cB(\hat F)\,,
\\
&\d_{\hat t}\hat f\!+\!\frac{V_g}{V_p}\hat w\!\cdot\!\grad_{\hat x}\hat f=\cN_pS_{pg}L\frac{V_g}{V_p}\hat\cR(\hat f,\hat F)+\cN_gS_{gg}L\frac{V_g}{V_p}\hat\cC(\hat f)\,.
\ea
\right.
\ee

Throughout the present study, we shall always assume that 
\be\lb{NoppColl}
\cN_pS_{pp}L\ll 1\,.
\ee
In other words, the collision integral for dust particles or droplets $\cN_pS_{pp}L\hat\cB(\hat F)$ is considered as formally negligible, and will be henceforth systematically discarded in the equations.

Besides, the thermal speed $V_p$ of dust particles or droplets is in general smaller than the thermal speed $V_g$ of gas molecules; thus we denote their ratio by
\be\lb{Def-eps}
\eps=\frac{V_p}{V_g}\in[0,1]\,.
\ee

Recalling that the mass ratio $[0,1]\ni \eta = m_g/m_p$ is supposed to be extremely small, since the particles are usually much bigger than the molecules, we also assume 
\be\lb{Def-eta}
\eta=\frac{\cN_p}{\cN_g}\in[0,1]\,.
\ee
This assumption on the ratio of the number of particles to the number of molecules defines a scaling such that the mass density of the gas is of the same order of magnitude as the mass density of droplets.

Finally, we shall assume that 
\be\lb{Scal-pg-gg}
\cN_p\,S_{pg}\, L = \eps\,,\quad\hbox{ and }\quad\cN_g\,S_{gg}\, L = 1/\eps\,.
\ee

Under these assumptions, 
$$
\ba
\cN_gS_{pg}L\frac{V_g}{V_p}\!=\!\frac{\cN_g}{\cN_p}(\cN_pS_{pg}L)\frac{V_g}{V_p}\!=\!\frac1\eta\,,
\\	\\
(\cN_pS_{pg}L)\frac{V_g}{V_p}\,=\,(\cN_gS_{gg}L)\frac{V_g}{V_p}=1\,,
\ea
$$
so that we arrive at the scaled system
\be\lb{BoltzSysSc}
\left\{
\ba
{}&\d_{\hat t}\hat F\,+\,\hat v\cdot\grad_{\hat x}\hat F\,=\frac{1}{\eta}\hat\cD(\hat F,\hat f)\,,
\\
&\d_{\hat t}\hat f+\frac{1}{\eps}\hat w\cdot\grad_{\hat x}\hat f=\hat\cR(\hat f,\hat F)+\frac{1}{\eps^2}\hat\cC(\hat f)\,.
\ea
\right.
\ee

Henceforth, we drop hats on all dimensionless quantities and variables introduced in this section. Only dimensionless variables, distribution functions and collision integrals will be considered from now on.  We also use $V,W$ as dummy
variables in the gain part of the collision operators $\cD$ and $\cR$, in order to avoid confusion.

We define therefore the ($\eps$- and $\eta$-dependent) dimensionless collision integrals
\be\lb{newc}
\ba 
\cC(f)( w)=&\iint_{\bR^3\times\bR^3}f(w') f(w'_*) \Pi_{gg}(w,\dd w'\,\dd w'_*)
\\ 
&- f(w)\int_{\bR^3} f(w_*)| w- w_*|  \Si_{gg}(|w- w_*|)\,\dd w_*\,,
\ea 
\ee 
\be\lb{newd}
\ba 
\cD( F, f)( v)=&\iint_{\bR^3\times\bR^3} F(V)f(W)\Pi_{pg}( v,\dd V\,\dd W)
\\
&- F(v)\int_{\bR^3} f(w)\left|\eps v- w\right|\Si_{pg}\left(\left|\eps v- w\right|\right)\,\dd w\,,
\ea
\ee 
\be\lb{newr}
\ba 
 \cR( f, F)( w)=&\iint_{\bR^3\times\bR^3} F(V) f(W) \Pi_{gp}( w,\dd V\,\dd W)
\\
&- f( w)\int_{\bR^3} F(v)\left|\eps v- w\right| \Si_{pg}\left(\left|\eps v-w\right|\right)\,\dd  v\,,
\ea
\ee 
with $\Si_{gg}$, $\Si_{pg}$ satisfying (\ref{Colli2}). Notice that the scattering kernels $\Pi_{pg}$ and $\Pi_{gp}$ depend in fact on $\eps$ and $\eta$. Whenever necessary (for instance in describing the asymptotic behavior of these kernels
in the small $\eps$ and $\eta$ limit), we shall denote them $\Pi_{pg}^{\eps,\eta}$ and $\Pi_{gp}^{\eps,\eta}$ respectively.  

The scaled Boltzmann system (\ref{BoltzSysSc}) is then recast as 
\be\lb{BoltzSysSc2}
\left\{
\ba
{}&\d_t F\,+\,v\cdot\grad_x F\,=\frac{1}{\eta}\cD(F,f)\,,
\\
&\d_t f+\frac{1}{\eps}w\cdot\grad_x f=\cR(f, F)+\frac{1}{\eps^2}\cC(f)\,.
\ea
\right.
\ee

\subsection{Explicit formulas for the collision integrals}\lb{SS-ExplctCollInt}


In the previous section, we have introduced a general setting for the various collisional processes involved in gas-particle mixtures. The explicit formulas for the main examples of collision integrals considered in this work are given
in the next three sections.

\subsubsection{The Boltzmann collision integral for gas molecules}


The dimensionless collision integral $\cC(f)$ is given by the formula
\be\label{cc1}
\cC(f)(w)=\iint_{\bR^3\times\bS^2}(f(w')f(w'_*)-f(w)f(w_*))c(w-w_*,\om)\,\dd w_*\dd \om ,
\ee
for each measurable $f$ defined a.e. on $\bR^3$ and rapidly decaying at infinity, where 
\be\label{cc2}
\ba
w'\equiv\,w'(w,w_*,\om):=w\,-(w-w_*)\cdot\om\om\,,
\\
w'_*\equiv\!w'_*(w,w_*,\om):=w_*\!+(w-w_*)\cdot\om\om\,,
\ea
\ee
(see formulas (3.11) and (4.16) in chapter II of \cite{Cerci75}).
 The collision kernel $c$ is of the form
\be\label{cc3}
c(w-w_*,\om)=|w-w_*|\si_{gg}(|w-w_*|,|\cos(\widehat{w-w_*,\om})|),
\ee
where $\si_{gg}$ is the dimensionless differential cross-section of gas molecules. In other words, 
$$
\Si_{gg}(|z|)=4\pi\int_0^1\si_{gg}(|z|,\mu)\,\dd\mu\,,
$$
while
\be
\label{Pigg}
\Pi_{gg}(w,\cdot)=\iint_{\bR^3\times\bS^2}\dd w_*\dd \om \,\de_{w'(w,w_*,\om)}\otimes\de_{w'_ *(w,w_*,\om)}c(w-w_*,\om)
\ee
The left hand side is to be understood a function of $w$ with values in the set of positive Borel measures on $\bR^3\times\bR^3$, while the right hand side is a linear superposition of the positive Borel measures 
$\de_{w'(w,w_*,\om)}\otimes\de_{w'(w,w_*,\om)}$ on $\bR^3\times\bR^3$ obtained by integrating over $w_*,\om$ while $w$ is kept fixed. 

If more to one's taste, one can equivalently formulate this equality by applying both sides to a test function $\chi\in C_c(\bR^3\times\bR^3)$: 
$$ 
\ba 
\iint_{\bR^3\times\bR^3}\chi(W,W_*)\Pi_{gg}(w,\dd W\dd W_*)
\\
=\iint_{\bR^3\times\bS^2}\la\de_{w'(w,w_*,\om)}\otimes\de_{w'_*(w,w_*,\om)},\chi\ra c(w-w_*,\om)\,\dd w_*\dd\om& 
\\ 
=\iint_{\bR^3\times\bS^2}\chi(w'(w,w_*,\om),w'_*(w,w_*,\om))c(w-w_*,\om)\,\dd w_*\dd\om&\,. 
\ea 
$$ 
One recognizes in the last right hand side of the equalities above the usual expression for the gain term in the Boltzmann collision integral for identical particles interacting by elastic collisions.

We recall that the collision integal $\cC$ satisfies the conservation of mass, momentum and kinetic energy (\ref{star}) --- see formulas (1.16)-(1.18) in chapter II of \cite{Cerci69}.

We assume that the molecular interaction is defined in terms of a hard potential satisfying Grad's cutoff assumption. In other words, we assume that there exists $c_*>1$ and $\g\in[0,1]$ such that
\be\label{cc4}
\ba
{}&0\,<\,c(z,\om)\,\le\, c_*(1+|z|)^\g\,,&&\quad\hbox{ for a.e. }(z,\om)\in\bR^3\times\bS^2\,,
\\
&\int_{\bS^2}c(z,\om)\,\dd\om\ge\frac1{c_*}\frac{|z|}{1+|z|}\,,&&\quad\hbox{ for a.e. }z\in\bR^3\,.
\ea
\ee

Next we discuss the properties of the linearization about a Maxwellian equilibrium state of the collision integral $\cC$. By scaling and Galilean invariance, one can consider the Maxwellian distribution
\be\lb{maxw}
M(w):=\tfrac1{(2\pi)^{3/2}}e^{-|w|^2/2}
\ee
without loss of generality. The linearized collision integral is defined as
\be\label{defL}
\cL\phi:=-M^{-1}D\cC(M)\cdot(M\phi)\,,
\ee
where $D$ is the functional derivative.

The following result is a theorem of Hilbert in the case of hard sphere collisions, extended by Grad to the case of hard cutoff potentials (see \cite{Cerci75}, especially Theorem I on p.186 and Theorem II on p.187).

\begin{theorem}\lb{T-HilbGr}
The linearized collision integral $\cL$ is an unbounded operator on $L^2(M\dd v)$ with domain $\Dom\cL=L^2((\bar c\star M)^2 M\dd v)$, where 
$$
\bar c(z):=\int_{\bS^2}c(z,\om)\dd\om\,.
$$
Moreover, $\cL=\cL^*\ge 0$, with nullspace
\be\label{null}
\Ker\cL=\Span\{1,w_1,w_2,w_3,|w|^2\}.
\ee
Finally, $\cL$ is a Fredholm operator, so that 
$$
\Img\cL=\Ker\cL^\bot\,.
$$
\end{theorem}

Defining by 
\be\label{defA}
A(w):=w\otimes w-\tfrac13|w|^2I
\ee
the traceless component of the tensor $w\otimes w$, we see that $A\bot\Ker\cL$ in $L^2(M\dd v)$. Since $\cL$ satisfies the Fredholm alternative, there exists a unique $\tilde A\in\Dom\cL$ such that 
\be\label{defAtilde}
\cL\tilde A=A,\quad \quad\tilde A\bot\Ker\cL\,.
\ee
Using the symmetry properties of the collision integral and the rotation invariance of the Maxwellian distribution (\ref{maxw}), one can show that the matrix field $\tilde A$ is of the form 
\be\label{defalpha}
\tilde A(w)=\a(|w|)A(w)\,,
\ee
where $\a$ is a measurable function such that 
$$
\int_{\bR^3}\a(|w|)^2|w|^4(\bar c\star M(w))^2 M(w)\dd w<\infty\,.
$$
See \cite{dego} for a complete proof of this statement. 

In the sequel, we shall assume for simplicity that the molecular interaction is such that
$$
\a\in L^\infty(\bR_+)\,.
$$
It is a well known fact that, in the case of Mawxell molecules, that is, in the case where the collision kernel is of the form
$$
c(z,\om)=C(|\cos(\widehat{v-v_*,\om})|)\,,
$$ 
then $\a$ is a positive constant. (See for instance the discussion between formulas (3.15) and (3.17) in chapter V of \cite{Cerci69}.)

\subsubsection{The collision integrals $\cD$ and $\cR$ for elastic collisions}\label{sec232}


For each measurable $F$ and $f$ defined a.e. on $\bR^3$ and rapidly decaying at infinity, the dimensionless collision integrals $\cD(F,f)$ and $\cR(f,F)$ are given by the formulas
$$
\ba
\cD(F,f)(v)&=\iint_{\bR^3\times\bS^2}(F(v'')f(w'')\!-\!F(v)f(w))b(\eps v-w,\om)\,\dd w\dd\om\,,
\\
\cR(f,F)(w)&=\iint_{\bR^3\times\bS^2}(f(w'')F(v'')\!-\!f(w)F(v))b(\eps v-w,\om)\,\dd v\dd\om\,,
\ea
$$
where
\be\label{ela1}
\ba
{}&v''\equiv v''(v,w,\om)\,:=v\,-\frac{2\eta}{1+\eta}\!\left(v-\!\frac1\eps w\!\right)\!\cdot\om\om\,,
\\
&w''\!\equiv w''(v,w,\om)\!:=w-\frac{2}{1+\eta}\,\,(\,w-\eps v)\cdot\om\om\,,
\ea
\ee 
(see formula (5.10) in chapter II of \cite{Cerci75}). The collision kernel $b$ is of the form
\be\label{ela2}
b(\eps v-w,\om)=|\eps v-w|\si_{pg}(|\eps v-w|,|\cos(\widehat{\eps v-w,\om})|),
\ee
where $\si_{pg}$ is the dimensionless differential cross-section of gas molecules. In other words, 
\be\label{ela3}
\Si_{pg}(|z|)=4\pi\int_0^1\si_{pg}(|z|,\mu)\,\dd\mu\,,
\ee
while
\be\label{ela4}
\ba
\Pi_{pg}(v,\cdot)&=\iint_{\bR^3\times\bS^2}\dd w\,\dd\om\,b(\eps v-w,\om)\de_{v''(v,w,\om)}\otimes\de_{w''(v,w,\om)}\,,
\\
\Pi_{gp}(w,\cdot)&=\iint_{\bR^3\times\bS^2}\dd v\,\dd\om\,b(\eps v-w,\om)\de_{v''(v,w,\om)}\otimes\de_{w''(v,w,\om)}\,,
\ea
\ee
where the equalities (\ref{ela4}) are to be understood in the same way as (\ref{Pigg}).

One should keep in mind that the velocity of each species is measured in units of the thermal speed of that species. This accounts for the appearance of the thermal speed ratio $\eps$ in the formulas above. Moreover, the reduced mass of 
the dust particles or droplets and gas molecules  defined by formula (5.2) in chapter II of \cite{Cerci75} is
$$
\frac{m_pm_g}{m_p+m_g}=\frac{m_g}{1+\eta}=\frac{m_p\eta}{1+\eta}\,.
$$
These formulas explain how the mass ratio $\eta$ appears in the definition of $v''$ and $w''$ above.

We recall that the operators $\cD$ and $\cR$ defined in this subsection satisfy separately the conservation of the number of particles and molecules (\ref{starstar}), and jointly the conservation of momentum (involving both operators):
\be\label{momDR}
\eps \int_{\bR^3}\cD(F,f)(v)v\,\dd v+ \eta \int_{\bR^3}\cR(f,F)(w)w\,\dd w=0\,.
\ee
This last identity is a dimensionless version of (\ref{troispr}).

These properties can be easily checked using the formulas
\be\label{symcol}
\eps v'' + \eta w''= \eps v + \eta w\,,\quad\eps v'' - w''= R_{\omega} (\eps v - w), 
\ee
where $R_{\omega}$ is the reflection defined by $R_{\omega} w = w - 2 (w \cdot\omega) \omega$ for each $\om\in\bS^2$. Indeed these formulas show that $(v,w) \mapsto (v'', w'')$ is a linear involution for each $\om\in\bS^2$.

As in the case of the molecular collision kernel $c$, we assume that $b$ is a cutoff kernel associated with a hard potential, i.e. we assume that there exists $b_*>1$ and $\b^*\in[0,1]$ such that
\be\label{ela5}
\ba
{}&0<b(z,\om)\le b_*(1+|z|)^{\b^*}\,,&&\quad\hbox{ for a.e. }(z,\om)\in\bR^3\times\bS^2\,,
\\
&\int_{\bS^2}b(z,\om)\,\dd\om\ge\frac1{b_*}\frac{|z|}{1+|z|}\,,&&\quad\hbox{ for a.e.}z\in\bR^3\,.
\ea
\ee
We also assume that (for a.e. $\omega \in \cS^2$) 
\be\lb{bC1}
b(\cdot,\om)\in C^1(\bR^3\setminus\{0\})\,,\hbox{ and }\sup_{\om\in\bS^2}|\d_zb(z,\om)|\le C (1+ |z|)\,.
\ee

\subsubsection{An inelastic model of collision integrals $\cD$ and $\cR$}\label{sec233}


Dust particles or droplets are macroscopic objects when compared to gas molecules. This suggests using the classical models of gas-surface interaction to describe the impingement of gas molecules on dust particles or droplets. Perhaps
the simplest such model of collisions has been introduced by F. Charles in \cite{FCharlesRGD08}, with a detailed discussion in section 1.3 of \cite{FCharlesPhD} and in \cite{FCharlesSDelJSeg}. We briefly recall this model below.

First, the (dimensional) particle-molecule cross-section is
$$
S_{pg}=\pi(r_g+r_p)^2,
$$
where $r_g$ is the molecular radius and $r_p$ the radius of dust particles or droplets. Then, the dimensionless particle-molecule cross-section is
$$
\Si_{pg}(|\eps v-w|)=1\,.
$$
The formulas for $S_{pg}$ and $\Si_{pg}$ correspond to a binary collision between two balls of radius $r_p$ and $r_g$.

Next, the measure-valued functions $\Pi_{pg}$ and $\Pi_{gp}$ are defined as follows:
\be\label{ine1}
\ba
\Pi_{pg}(v,\dd V\,\dd W)&:=K_{pg}(v,V,W)\,\dd V\dd W\,,
\\
\Pi_{gp}(w,\dd V\,\dd W)&:=K_{gp}(w,V,W)\,\dd V\dd W\,,
\ea
\ee 
where, 
\be\label{ine2}
\ba
K_{pg}(v,V,W):&=\tfrac1{2\pi^2}\left(\tfrac{1+\eta}\eta\right)^4\b^4\eps^3\exp\left(-\tfrac12\b^2\left(\tfrac{1+\eta}\eta\right)^2\left|\eps v-\frac{\eps V+\eta W}{1+\eta}\right|^2\right)
\\
&\qquad\times\int_{\bS^2}(n\cdot(\eps V-W))_+\left(n\cdot\left(\frac{\eps V+\eta W}{1+\eta}-\eps v\right)\right)_+dn ,
\ea
\ee
\be\label{ine3}
\ba
K_{gp}(w,V,W)&:=\tfrac1{2\pi^2}(1+\eta)^4\b^4\exp\left(-\tfrac12\b^2(1+\eta)^2\left|w-\frac{\eps V+\eta W}{1+\eta}\right|^2\right)
\\
&\qquad\times\int_{\bS^2}(n\cdot(\eps V-W))_+\left(n\cdot\left(w-\frac{\eps V+\eta W}{1+\eta}\right)\right)_+dn .
\ea
\ee
In these formulas
$$
\beta=\sqrt{\frac{m_g}{2 k_B T_{surf}}}
$$ 
where $k_B$ is the Boltzmann constant and $T_{surf}$ the surface temperature of the particles.

Thus, defining
\be\lb{DefPl}
P[\l](\xi,n):=\tfrac1{2\pi}\l^4\exp(-\tfrac12\l^2|\xi|^2)(\xi\cdot n)_+ ,
\ee
for each $\l>0$ and $n\in\bS^2$, we see that the integral kernels $K_{pg}$ and $K_{gp}$ are given in terms of $P$ by the expressions
$$
\ba
K_{pg}(v,V,W)=&\tfrac1\pi\eps^3\int P[\b\tfrac{1+\eta}\eta]\left(\tfrac{\eps V+\eta W}{1+\eta}-\eps v,n\right)((\eps V-W)\cdot n)_+dn,
\\
K_{gp}(w,V,W)=&\tfrac1\pi\int P[\b(1+\eta)]\left(w-\tfrac{\eps V+\eta W}{1+\eta},n\right)((\eps V-W)\cdot n)_+dn .
\ea
$$


\section{Assumptions on $\Pi_{pg}$ and $\Pi_{gp}$}\lb{S-3}


In the sequel, we shall state a theorem which holds for all collision integrals satisfying a few assumptions introduced below. 

We recall that $\Pi_{pg}$ and $\Pi_{gp}$ are nonnegative measure-valued functions of the variable $v\in\bR^3$ and $w\in\bR^3$ resp., which depend in general on the small parameters $\eps$ and $\eta$ (see formulas (\ref{ela1}), (\ref{ela4}), 
and (\ref{ine1})-(\ref{ine3})). We do not make this dependence explicit, unless if necessary (as in Assumptions (H4)-(H5) below). In this case, we write $\Pi_{pg}^{\eps, \eta}$ and $\Pi_{gp}^{\eps, \eta}$ instead of $\Pi_{pg}$ and $\Pi_{gp}$.

\bigskip
\noindent
\textbf{Assumption (H1).} There exists a nonnegative measurable function 
$$
q\equiv q(r)\le C(1+r)\quad\hbox{ for some }C>0
$$
such that the measure-valued functions $\Pi_{pg}$ and $\Pi_{gp}$ satisfy
$$
\int_{\bR^3}\Pi_{pg}(v,\dd V\dd W)\dd v=\int_{\bR^3}\Pi_{gp}(w,\dd V\dd W)\,\dd w=q(|\eps V-W|)\dd V\dd W\,.
$$
Note that Assumption (H1) is coherent with the fact that in the last two lines of (\ref{Colli2}), the same cross-section $\Si_{pg}$ appears (and thus with the conservation of mass).

\bigskip
\noindent
\textbf{Assumption (H2).} There exists a function $Q\equiv Q(r)\in C(\bR_+^*)$ satisfying
$$
Q\ge 0\,,\quad\hbox{ and }Q(r) + |Q'(r)|\le C(1+r)\hbox{ for some }C>0\,,
$$
such that the measure-valued functions $\Pi_{pg}$ and $\Pi_{gp}$ satisfy
$$
\ba
\eps\int_{\bR^3}\dd v\,(v-V)\Pi_{pg}(v,\dd V\dd W)&=-\eta\int_{\bR^3}\dd w\,(w-W)\Pi_{gp}(w,\dd V\dd W)
\\
&=-\frac{\eta}{1+\eta}(\eps V-W)Q(|\eps V-W|)\dd V\dd W\,.
\ea
$$
This assumption implies the conservation of momentum between molecules and particles.

\bigskip
\noindent
\textbf{Assumption (H3).} There exists a constant $C>0$ such that the measure-valued function $\Pi_{pg}$ satisfies
$$
\int_{\bR^3}\dd v\,\left|\eps v-\frac{\eps V+\eta W}{1+\eta}\right|^2\Pi_{pg}(v,\dd V\dd W)\le C\,\eta^2\,(1+|\eps V-W|^2)q(|\eps V-W|)\dd V\dd W ,
$$
where $q$ is the function appearing in Assumption (H1).

\bigskip
\noindent
\textbf{Assumption (H4).} The limiting measure $\Pi^{0,0}_{gp}$ satisfies the following invariance\footnote{The notation $\cT\# m$ designates the push-forward of the measure $m$ by the transformation $\cT$.} property:
$$
\cT_R\#\Pi^{0,0}_{gp}=\Pi^{0,0}_{gp}\quad\hbox{ for each }R\in O_3(\bR)\,,
$$
where
\be\label{defTR}
\cT_R:\,(w,V,W)\mapsto(Rw,V,RW)\,.
\ee 
Besides, for each  $\Phi := \Phi(w,W)$ such that $|\Phi(w,W)|\le C (1+|w|^2+|W|^2)M(W)$, 
$$
\int_{\bR^3}(1+ |V|^2)^{-p}\left|\iint_{\bR^3\times\bR^3}\Phi(w,W)(\Pi^{\eps,\eta}_{gp}(w,\dd V\dd W)-\Pi^{0,0}_{gp}(w,\dd V\dd W))\,\dd w\right|\to 0
$$
for some $p>3$, as $\eps,\eta\to 0$. Moreover, 
$$
\iint_{\bR^3\times\bR^3}\dd w\,(1+|w|^2+|W|^2)M(W)\Pi^{0,0}_{gp}(w,dVdW)\in L^1((1+V^2)^{-3} dV)\,.
$$

\bigskip
\noindent
\textbf{Assumption (H5).} For all $h \in L^2(M(w)\dd w)$, 
$$
\ba
\iiint_{\bR^3\times \bR^3\times\bR^3}(1\!+\!|W|^2)(1\!+\!|V|^2)^{-p}(1\!+\!|w|^2)M(W)|h(W)|\Pi^{\eps,\eta}_{gp}(w,\!\dd V\!\dd W)\,\dd w&
\\
\le C ||h||_{L^2(M(w)\dd w)}&, 
\ea
$$
where $C$ does not depend on $\eta$ and $\eps$ (for $\eta$ and $\eps$
close to $0$).

\bigskip
We next prove that the elastic and inelastic models previously introduced (in sections \ref{sec232} and \ref{sec233} resp.) satisfy the assumptions (H1)-(H5). 

\subsection{Verification of (H1)-(H5) for the elastic collision model}


\begin{proposition}\label{hypelast}
For each collision kernel $b$ of the form (\ref{ela2}) satisfying (\ref{ela5}), let the quantities $\Sigma_{pg}$, $\Pi_{pg}$ and $\Pi_{gp}$ be defined by (\ref{ela1}), (\ref{ela3}) and (\ref{ela4}). Then, assumptions (H1)-(H5) are satisfied, with 
\be\label{H11}
q(|\eps v-w|)=4\pi\int_0^1 |\eps v-w|\,\sigma_{pg}(|\eps v-w|,\mu)\,\dd\mu\,,
\ee
\be\label{H12} 
Q(|\eps v-w|)=8\pi\int_0^1 |\eps v-w|\,\sigma_{pg}(|\eps v-w|,\mu)\mu^2\,\dd\mu\,,
\ee
and
\be\label{H13}
C = 1.
\ee
\end{proposition} 

\begin{proof}
For each continuous and compactly supported test functions $\phi\equiv\phi(v,V,W)$ and $\psi\equiv\psi(w,V,W)$, one has
\be\lb{for1}
\ba
\iiint_{\bR^3\times\bR^3\times\bR^3}\phi(v,V,W) \Pi_{pg}(v,\dd V\dd W)\,\dd v& 
\\
=\iiint_{\bR^3\times\bR^3\times\bS^2}\phi(v,v'',w'')b(\eps v - w, \om)\,\dd\om\dd w\dd v&
\\
= \iiint_{\bR^3\times\bR^3\times\bS^2}\phi(v'',v,w) b(\eps v - w, \omega)\,\dd\om\dd w\dd v&\,.
\ea
\ee
where the last equality follows from the fact that the map $(v,w) \mapsto (v'', w'')$ is a linear involution for each $\om\in\bS^2$. By the same token
\be\lb{for2}
\ba
\iiint_{\bR^3\times\bR^3\times\bR^3}\psi(w,V,W)\Pi_{gp}(w,\dd V\dd W)\dd w
\\
=\iiint_{\bR^3\times\bR^3\times\bS^2}\psi(w'',v,w) b(\eps v - w, \om) \dd\om\dd w\dd v&\,.
\ea
\ee
Observing that 
$$
\int_{\bS^2}b(\eps v-w,\om)\,\dd\om=4\pi|\eps v-w|\int_0^1\sigma_{pg}(|\eps v-w|,\mu)\,\dd\mu\,, 
$$
one arrives at assumption (H1) with $q$ defined by (\ref{H11}).
 
Then we see that 
$$
\ba
\eps(v-v'')\cdot\om&=-\eta(w-w'')\cdot\om
\\
&= -\frac{2\eta}{1+\eta}(w-\eps v)\cdot\om=\frac{2\eta}{1+\eta}(w''-\eps v'')\cdot\om\,,
\ea
$$
and that
$$ 
\ba
\int_{\bS^2}\left((\eps v'' - w'')\cdot\om\right)b(\eps v - w,\om)\om\,\dd\om&
\\
=4\pi|\eps v-w|(\eps v''-w'')\int_0^1\sigma_{pg}(|\eps v-w|,\mu)\mu^2\,\dd\mu&\,,
\ea
$$
and conclude that assumption (H2) holds with $Q$ defined by (\ref{H12}).
 
Observing that 
$$ 
\ba
\bigg| \eps v - \frac{\eps v'' + \eta w''}{1 + \eta}\bigg|^2=&\left(\frac{\eta}{1+\eta}\right)^2\, |\eps v'' -w''|^2
\\
=&\left(\frac{\eta}{1+\eta}\right)^2\,|\eps v -w|^2 \le \eta^2\,  |\eps v -w|^2\,,
\ea
$$
shows that assumption (H3) holds with $C=1$.

Next, one has
$$
\ba
\iiint_{\bR^3\times\bR^3\times\bR^3}\phi(w,V,W)\Pi^{0,0}_{gp}(w,\dd V\dd W)\,\dd w&
\\
=\iiint_{\bR^3\times\bR^3\times\bS^2}\phi(w-2w\cdot\om\om,v,w)b(w,\om)\,\dd v\dd w\dd\om&\,,
\ea
$$ 
which obviously implies the relation
$$
\Pi^{0,0}_{gp}=T_R\#\Pi^{0,0}_{gp}\,.
$$
Besides, for each $p>3$ and each continuous $\Phi\equiv\Phi(w,W)$ such that 
$$
|\Phi(w,W)| \le C (1+|w|^2+|W|^2)M(W)
$$
one has
$$
\ba
\int_{\bR^3}(1+ |V|^2)^{-p}\left|\iint_{\bR^3\times\bR^3}\Phi(w,W)(\Pi^{\eps,\eta}_{gp}(w,\dd V\dd W)-\Pi^{0,0}_{gp}(w,\dd V\dd W))\,\dd w\right|&
\\
=\!\!\!\int_{\bR^3}\!\!\!(1+ |v|^2)^{-p}\left|\iint_{\bR^3\times\bS^2}\!\!\left(\Phi(w'',w)b(\eps v\! -\! w,\om)\!-\!\Phi(\tilde w,w)b(w,\omega)\right)\dd v\dd w\dd\om\right|&\,,
\ea
$$
where
$$
\tilde w=w-2(w\cdot\om)\om\,.
$$
By continuity of $b$ and $\Phi$, we see that
$$
\Phi(w'',w)b(\eps v-w,\om)\to\Phi(\tilde w,w)b(w,\omega)
$$
as $\eps,\eta\to 0$. Then, using the estimate 
$$ 
|\Phi(w'',w)|b(\eps v - w,\omega)\le C(1+|w|^2)^{3/2}(1 + |v|^2)^{3/2}M(w)\,,
$$
we conclude that
$$
\int_{\bR^3} (1+ |V|^2)^{-p} \left|\iint\Phi(w,W)(\Pi^{\eps,\eta}_{gp}(w,dVdW)dw-\Pi^{0,0}_{gp}(w,dVdW)dw)\right| \to 0
$$   
as $\eps,\eta\to 0$ by dominated convergence. Finally we observe that 
$$
\ba
\iint_{\bR^3\times\bR^3\times\bR^3}(1+|w|^2 + |W|^2)(1+V^2)^{-3}M(W)\Pi^{0,0}_{gp}(w,\dd V\dd W)\,\dd w&
\\
= \iiint_{\bR^3\times\bR^3\times\bS^2} (1+|w-2(w\cdot\om)\om|^2+|w|^2)(1+|v|^2)^{-3}M(w)b(-w,\om)\,\dd\om\dd w\dd v&
\\
\le C\iint_{\bR^3\times\bR^3}(1+|w|^2) (1+v^2)^{-3}M(w)|w|\,\dd w\dd v<\infty&\,,
\ea
$$
so that assumption (H4) is satisfied.

Finally, for each $h\in L^2(M(w)\,\dd w)$ and each $\eps, \eta>0$ small enough,
$$  
\ba
\iiint_{\bR^3\times\bR^3\times \bR^3}(1\!+\!|W|^2)(1\!+\!|V|^2)^{-p}(1\!+\!|w|^2)M(W)|h(W)|\Pi^{\eps,\eta}_{gp}(w,\dd V\dd W)\,\dd w&
\\
=\iiint_{\bR^3\times \bR^3\times\bS^2}(1\!+\!|w|^2)(1\!+\!|v|^2)^{-p} (1\!+\!|w''|^2)M(w)|h(w)|b(\eps v-w,\om)\,\dd\om\dd w\dd v& 
\\
\le C\iint_{\bR^3\times \bR^3}(1\!+\!|w|^2)(1\!+\! |v|^2)^{-p}(1\!+\!|v|^2\!+\!|w|^2)(1\!+\!|v|\!+\!|w|)M(w)|h(w)|\,\dd w\dd v&
\\
\le C\int_{\bR^3}(1+|w|^2)^{5/2}M(w)|h(w)|\dd w\le C  ||h||_{L^2(M(w)\dd w)}&
\ea
$$
by the Cauchy-Schwarz inequality. Hence assumption (H5) is also verified.
\end{proof}

\subsection{Verification of (H1)-(H5) for the inelastic collision model}


\begin{proposition}\label{hypinelast}
The scattering kernels $\Pi_{pg}$ and $\Pi_{gp}$ defined by (\ref{ine1})-(\ref{ine3}) satisfy assumptions (H1)-(H5), with 
$$ 
q(|\eps v-w|)=|\eps v-w|\,,\quad Q(|\eps v-w|)=\frac{\sqrt{2\pi}}{3\b}+|\eps v-w|\,,
$$
and 
$$ 
C = \frac{16}{\beta^2}\,.
$$
\end{proposition}

As in the previous section, we explicitly mention the $\eps,\eta$-dependence of the scattering kernels whenever needed, in which case we use the notation $K^{\eps,\eta}_{pg},K^{\eps,\eta}_{gp}$ to designate $K_{pg}$ and $K_{gp}$ respectively.

\begin{proof}
Setting successively $a=\frac{\eps V+ \eta W}{1+\eta}-\eps v$ and $b =\beta\left(\frac{1+\eta}{\eta}\right)a$, we see that 
$$
\ba
\int_{\bR^3}K_{pg} (v,V,W)\,\dd v&
\\
=\iint_{\bR^3\times\bS^2}\frac{\beta^4}{2\pi^2}\left(\tfrac{1+\eta}{\eta}\right)^4\exp\left(-\tfrac12\beta^2\left(\tfrac{1+\eta}{\eta}\right)^2|a|^2\right)(a \cdot n)_+ \left((\eps  V-W)\cdot n\right)_+\,\dd a\dd n&
\\
=\frac1{2\pi^2}\iint_{\bR^3\times\bS^2}\exp(-\tfrac12|b|^2)(b \cdot n)_+\left((\eps  V- W)\cdot n\right)_+\,\dd b\dd n&
\\
=\tfrac1{\pi}\int_{\bS^2}\left((\eps  V- W)\cdot n\right)_+\,\dd n=|\eps V- W|&\,. 
\ea
$$

Likewise, setting $b'=\beta(1+\eta)\left(w-\frac{\eps V+ \eta W}{1+\eta}\right)$, we see that
$$ 
\ba
\int_{\bR^3}K_{gp} (w,V,W)\,\dd w&
\\
=\frac1{2\pi^2}\iint_{\bR^3\times\bS^2}\exp(-\tfrac12|b'|^2)(b'\cdot n)_+\left((\eps  V- W)\cdot n\right)_+\,\dd b'\dd n=|\eps  V- W|&\,.
\ea
$$
so that assumption (H1) is verified.

With $b =\beta\left(\frac{1+\eta}{\eta}\right)\left(\frac{\eps V+ \eta W}{1+\eta}-\eps v\right)$ as above
$$ 
\ba
\eps\int_{\bR^3}(v-V)K_{pg} (v,V,W)\,\dd v&
\\
=\frac1\pi\frac{\eta}{1+\eta}\int_{\bS^2}\left((W-\eps  V)-\frac1{2\pi\beta}\int_{\bR^3}b e^{- |b|^2/2}(b \cdot n)_+\,\dd b\right)((\eps  V- W)\cdot n)_+\,\dd n&
\\
=-\frac1\pi\frac{\eta}{1+\eta}\left(\int_{\bS^2}(\eps V\! -\! W)\left((\eps  V\!-\! W)\cdot n\right)_+\,\dd n\!+\!\frac{\sqrt{2\pi}}{2\beta}\int n \left((\eps  V\!-\! W)\cdot n\right)_+\,\dd n\right)&
\\
= -\frac{\eta}{1+\eta}(\eps  V- W)\left(|\eps  V- W|+\frac{(2\pi)^{1/2}}{3\beta}\right)&\,.
\ea
$$
Likewise, setting $b'=\beta(1+\eta)\left(w-\frac{\eps  V+ \eta W}{1+\eta}\right)$ as above, we see that
$$ 
\ba
-\eta\int_{\bR^3}(w - W)K_{gp}(w,V,W)\,\dd w
\\
=-\frac1{\pi}\frac{\eta}{1+\eta}\int\left(\frac{(2\pi)^{1/2}}{2\beta}n+(\eps  V- W)\right)\left((\eps  V-W)\cdot n\right)_+\,\dd n 
\\
=-\frac{\eta}{1+\eta}(\eps V-W)\left(|\eps V-W|+\frac{ (2\pi)^{1/2}}{3\beta}\right) 
\ea
$$
so that assumption (H2) is also satisfied.

Still with $b=\beta\left(\frac{1+\eta}{\eta}\right)\left(\frac{\eps V+ \eta W}{1+\eta}-\eps v\right)$, one has
$$
\ba 
\int_{\bR^3}\left|\frac{\eps  V+ \eta W}{1+\eta}-\eps v\right|^2K_{pg}(v,V,W)\,\dd v& 
\\
=\left(\frac{\eta}{1+\eta}\right)^2\frac1{\pi\beta^2}\int_{\bS^2}\left(\frac{1}{2\pi}\int_{\bR^3}|b|^2\exp(-\tfrac12 |b|^2)(b\cdot n)_+\,\dd b\right)\,\dd n\,|\eps V-W|&
\\
\le\frac{16}{\beta^2}\left(\frac{\eta}{1+\eta}\right)^2|\eps  V- W|&\,,
\ea
$$ 
so that assumption (H3) is satisfied.

Observe that
$$
\Pi^{0,0}_{gp}(w,\dd V,\dd W)=K^{0,0}(w,W)\,\dd V\dd W
$$
with
$$
\ba
K^{0,0}(w,W)&=\tfrac1{2\pi^2}\b^4\exp(-\tfrac12\b^2|w|^2)\int_{\bS^2}(-W\cdot n)_+(w\cdot n)_+\,\dd n
\\
&=K^{0,0}(Rw,RW)
\ea
$$
for each $R\in O_3(\bR)$. Hence $\cT_R\#\Pi^{0,0}_{gp}=\Pi^{0,0}_{gp}$, which is the first property in (H4). 

On the other hand
$$
\ba
\iiint_{\bR^3\times\bR^3\times\bR^3}(1+|V|^2)^{-3}(1+|w|^2+|W|^2)M(W)\Pi^{0,0}_{gp}(w,\dd V,\dd W)\,\dd w 
\\
=\int_{\bR^3}\frac{\dd V}{(1+|V|^2)^3}\iint_{\bR^3\times\bR^3}(1+|w|^2+|W|^2)M(W)K^{0,0}(w,W)\,\dd w\dd W<\infty
\ea
$$
since
$$
0\le K^{0,0}(w,W)\le\tfrac2\pi\b^4|w||W|\exp(-\tfrac12\b^2|w|^2)\,.
$$
Hence the third property in (H4) is verified.

Let $\Phi\equiv\Phi(w,W)$ be such that $|\Phi(w,W)| \le C(1+|w|^2+|W|^2)M(W)$. Then, 
$$
\ba
\frac{2\pi^2}{\b^4}|\Phi(w,W)|K_{gp}^{\eps,\eta}(w,V,W)\le C (|V|+|W|+|w|)(|V|+|W|)&
\\
\times(1+|w|^2+|W|^2)M(W)\exp\left(-\tfrac12\beta^2\left|w-\frac{\eps V + \eta W}{1+\eta}\right|^2\right)&\,. 
\ea
$$
Since
$$ 
\ba
|W|^2+\beta^2\left|w-\frac{\eps}{1+\eta}V-\frac{\eta}{1+\eta}W\right|^2&
\\
\ge\min(1,\beta^2/2)\left(|W|^2+\left|w-\frac{\eta}{1+\eta}W\right|^2-|V|^2\right)&
\\
\ge\min(1,\beta^2/2)(|W|^2+|w|^2-2|V|^2)&\,,
\ea
$$
we see that 
$$
\ba
\frac{2\pi^2}{\b^4}|\Phi(w,W)|K_{gp}^{\eps,\eta}(w,V,W)\le C (|V|+|W|+|w|)(|V|+|W|)&
\\
\times(1+|w|^2+|W|^2)\exp(\mu|V|^2)\exp(-\tfrac12\mu(|w|^2+|W|^2))&\,,
\ea
$$
with
$$
\mu:=\min(1,\beta^2/2)\,.
$$
Therefore 
$$
\ba
\iint_{\bR^3\times\bR^3}\Phi(w,W)K_{gp}^{\eps,\eta}(w,V,W)\,\dd w\dd W
\\
\to\iint_{\bR^3\times\bR^3}\Phi(w,W)K_{gp}^{0,0}(w,V,W)\,\dd w\dd W
\ea
$$
for each $V\in\bR^3$ by dominated convergence. Setting $y=w-\frac{\eps V + \eta W}{1+\eta}$, one has
$$
\ba 
\frac{2\pi^2}{\b^4}(1\!+\!|V|^2)^{-p}\left|\iint_{\bR^3\times\bR^3}\Phi(w,W)K_{gp}^{\eps,\eta}(w,V,W)\,\dd w\dd W\right|
\\
=(1\!+\!|V|^2)^{-p}\left|\iint_{\bR^3\times\bR^3}\Phi\left(y\!+\!\frac{\eps V\!+\!\eta W}{1\!+\!\eta},W\right)(1\!+\!\eta)^4\exp\left(-\tfrac12\beta^2(1\!+\!\eta)^2|y|^2\right)\right.
\\
\left.\times\int_{\bR^3}\left((\eps V-W)\cdot n\right)_+(-y\cdot n)_+\,\dd n\dd y\dd W\right|
\\
\le C(1\!+\!|V|^2)^{-p}\iint_{\bR^3\times\bR^3}(1\!+\!|y|^2\!+\!|W|^2\!+\!|V|^2)M(W)\\
\\
\times\exp(-\tfrac12\beta^2|y|^2)|y|(|V|\!+\!|W|)\,\dd y\dd W
\ea
$$
which is integrable in $V\in\bR^3$ for $p>3$. Therefore,
$$
\int_{\bR^3} (1+ |V|^2)^{-p}\left|\iint\Phi(w,W)(\Pi^{\eps,\eta}_{gp}(w,\dd V\dd W)\,\dd w-\Pi^{0,0}_{gp}(w,\dd V\dd W)\,\dd w)\right|\to 0
$$
as $\eps,\eta\to 0$ for all $p>3$ by dominated convergence. This completes the verification of (H4).

Using again the substitution $y=w-\frac{\eps V + \eta W}{1+\eta}$, one has
$$
\ba
\iiint_{\bR^3\times \bR^3\times\bR^3}(1+|W|^2)M(W)|h(W)|(1+|V|^2)^{-p}(1+|w|^2)\Pi^{\eps,\eta}_{gp}(w,\dd V\dd W)\,\dd w&
\\
=\iiint_{\bR^3\times \bR^3\times\bR^3}(1+|W|^2)M(W)|h(W)|(1+|V|^2)^{-p}(1+|w|^2)&
\\
\times\frac{\beta^4}{2\pi^2}(1+\eta)^4\exp\left(-\tfrac12\beta^2(1+\eta)^2\left|w-\frac{\eps V + \eta W}{1+\eta}\right|^2\right)&
\\
\times\int_{\bS^2}(\eps V-W)\cdot n)_+\left(\left(\frac{\eps V+\eta W}{1+\eta}-w\right)\cdot n\right)_+\,\dd n\dd V\dd W\dd w&
\\
\le C \iiint_{\bR^3\times \bR^3\times\bR^3}(1+|W|^2)(1+ |V|^2)^{-p}(1+|V|^2+|W|^2+|y|^2)M(W)|h(W)|&
\\
\times|V+W||y|\exp\left(-\tfrac12\beta^2|y|^2\right)\,\dd V\dd W\dd y\le C\|h\|_{L^2(M(w)\,\dd w)}&\,,
\ea
$$
which is precisely assumption (H5).
\end{proof}


\section{Passage to the limit}\lb{S-4}


In this section, we use the material presented in sections \ref{S-2}-\ref{S-3} to state and prove the main result in this paper, i.e. the derivation of the Vlasov-Navier-Stokes model for thin sprays from the system of Boltzmann equations for a binary 
mixture of gas molecules and dust particles or droplets.

\subsection{Statement of the main result}


We henceforth consider a sequence of solutions $f_n\equiv f_n(t,x,w)$, and $F_n\equiv F_n(t,x,v)$ to the system of kinetic-fluid equations (\ref{BoltzSysSc2}), with sequences $\eps_n,\eta_n\to 0$ in the place of the parameters $\eps,\eta>0$:
\be\label{kifu}
\ba
{}&\d_tF_n+v\cdot\grad_xF_n=\frac1{\eta_n}\cD(F_n,f_n)\,,
\\
&\d_tf_n+\frac1{\eps_n}w\cdot\grad_xf_n=\cR(f_n,F_n)+\frac1{\eps_n^2}\cC(f_n)\,,
\ea
\ee
where $\cC$, $\cD$ and $\cR$ are defined by (\ref{cc1})-(\ref{cc3}), (\ref{newd}) and (\ref{newr}).

\begin{theorem}\lb{T-LimThm}
Assume that the scattering kernels $\Pi_{pg}^{\eps_n,\eta_n}$ and $\Pi_{gp}^{\eps_n,\eta_n}$ in (\ref{newd})-(\ref{newr}) satisfy (H1)-(H5), while the molecular collision kernel $c$ satisfies (\ref{cc4}). Assume further that the function $\a$ in 
(\ref{defalpha}) is bounded on $\bR_+$, and that 
$$
\eps_n\to 0\,,\quad\hbox{ and }\quad\eta_n/\eps_n^2\to 0\,.
$$

Let $g_n\equiv g_n(t,x,w)\ge 0$ and $F_n\equiv F_n(t,x,v)\ge 0$ be sequences of smooth (at least $C^1$) functions, and let
\be\label{kifu2}
f_n(t,x,w):=M(w)(1+\eps_n g_n(t,x,w)),
\ee 
where $M$ is the Maxwellian distribution (\ref{maxw}). Assume that 
$$
F_n\wto F\hbox{ in }L^\infty_{loc}\hbox{ weak-*}\,,\quad\hbox{ and that }g_n\wto g\hbox{ in }L^2_{loc}(\bR_+^* \times \bR^3 \times \bR^3)\hbox{ weak}
$$
for some $F\in L^\infty_{loc}(\bR_+\times\bR^3\times\bR^3)$ and $g\in L^2_{loc}(\bR_+^*\times\bR^3\times \bR^3)$.

\smallskip
Assume that

\noindent
(a) the pair $(F_n,f_n)$ is a solution to (\ref{kifu}), with $\cC,\cD,\cR$ defined by (\ref{cc1})-(\ref{cc3}), (\ref{newd}) and (\ref{newr})

\noindent
(b) there exists $p>3$ such that
$$
\sup_{n\ge 1}\sup_{(t,x,v)\in[0,R]\times[-R,R]^3\times\bR^3}(1+|v|^2)^pF_n(t,x,v)\le C_R<\infty
$$ 
for each $R>0$,

\noindent
(c) the sequence
$$
\int_{\bR^3}g_n(t,x,w)^2M(w)\,\dd w
$$
is bounded in $L^1_{loc}(\bR_+^* \times \bR^3)$,

\noindent
(d) the sequence of velocity averages of $g_n$
\be\label{cfor}
\int_{\bR^3}g_n\phi(w)M(w)\,\dd w\to\int_{\bR^3}g\phi(w)M(w)\,\dd w
\ee
strongly in $L^2_{loc}(\bR_+^* \times \bR^3)$ for each $\phi\in C_c(\bR^3)$.

\bigskip
Then there exist $L^{\infty}$ functions $\rho\equiv\rho(t,x)\in\bR$ and $\th\equiv\th(t,x)\in\bR$, and a $L^\infty$ vector field $u\equiv u(t,x)\in\bR^3$ s.t. 
\be \label{eqr}
\ba
g(t,x,w)=\rho(t,x)+u(t,x)\cdot w+\th(t,x)\tfrac12(|w|^2-3)&
\\
\hbox{ for a.e. }(t,x,w)\in\bR_+\times\bR^3\times\bR^3&\,,
\ea
\ee
and the pair $(F,u)$ satisfies the Vlasov-Navier-Stokes system
\be \label{VNS}
\left\{
\ba
{}&\d_tF+v\cdot\grad_xF=\ka\Div_v((v-u)F),
\\	
&\Div_xu=0,
\\	
&\d_tu+\Div_x(u\otimes u)=\nu\Dlt_xu-\grad_xp+\ka\int (v-u)F\,\dd v,
\ea
\right.
\ee
in the sense of distributions, with 
\be \label{nuka}
\nu:=\tfrac1{10}\int\tilde A:\cL\tilde A M(w)\,\dd w>0\,,\quad\ka:=\tfrac13\int Q(|w|)|w|^2M(w)\,\dd w>0,
\ee
where $Q$ is defined in assumption (H2), while $\tilde A,\cL$ are defined by (\ref{defAtilde})-(\ref{defL}).
\end{theorem}

\subsection{Proof of Theorem \ref{T-LimThm}}


\bigskip
The proof of Theorem \ref{T-LimThm} is based on the formal derivation of the incompressible fluid dynamic limit of the Boltzmann equation formulated in \cite{BGL1}. However, the interaction with the dust particles/droplets involves very serious
complications.

This proof is split in several steps, referred to as Propositions \ref{rhout} to \ref{last}, and a final part in which all the convergences of the different terms appearing in eq. (\ref{kifu}) are established.

\subsubsection{Step 1: Asymptotic form of the molecular distribution function.}


\begin{proposition}\label{rhout}
Under the assumptions of Theorem \ref{T-LimThm}, there exist $L^{\infty}$ functions $\rho\equiv\rho(t,x)\in\bR$ and $\th\equiv\th(t,x)\in\bR$, and a $L^\infty$ vector field $u\equiv u(t,x)\in\bR^3$ s.t. (\ref{eqr}) holds.
\end{proposition}

\begin{proof} Since $\cC$ is a quadratic operator, its Taylor expansion terminates at order $2$, i.e.
$$
\ba
\cC(M(1+\eps_ng_n))=&\cC(M)+\eps_nD\cC(M)\cdot(Mg_n)+\eps_n^2\cC(Mg_n)
\\
=&-\eps_nM\cL g_n+\eps_n^2M\cQ(g_n),
\ea
$$
where $\cL\phi$ is defined by (\ref{defL}) and
\be \label{defQ}
\cQ(\phi):=M^{-1}\cC(M\phi).
\ee
Then the kinetic equation for the propellant (second line of eq. (\ref{kifu})) can be recast in terms of the fluctuation of the distribution function $g_n$:
\be\label{eqg}
\d_tg_n+\frac1{\eps_n}w\cdot\grad_xg_n+\frac1{\eps_n^2}\cL g_n=\frac1{\eps_n}M^{-1}\cR(M(1+\eps g_n),F_n)+\frac1{\eps_n}\cQ(g_n)\,.
\ee
Multiplying each side of this equation by $\eps_n^2$ leads to the equality
\be \label{eqgpr}
\cL g_n=\eps_n(M^{-1}\cR(M(1+\eps g_n),F_n)+\cQ(g_n))-\eps_n^2\d_tg_n-\eps_nw\cdot\grad_xg_n .
\ee
The two last terms of this identity clearly converge to $0$ in the sense of distributions since $g_n \wto g$ weakly in $L^2_{loc}$. 

Next, for each test function $\phi \in C_c(\bR^3)$,
$$ 
\ba
\int_{\bR^3}\cQ(g_n)(w)\phi(w)\,\dd w=\iiint_{\bR^3\times\bR^3\times\bS^2}\left(M^{-1}(w')\phi(w')-M^{-1}(w)\phi(w)\right)&
\\
\times M(w_*)g_n(w_*)M(w)g_n(w) c(w - w_*, \om)\,\dd\om\dd w_*\dd w&\,,
\ea
$$
where $w',w'_*$ are defined by (\ref{cc2}). By the Cauchy-Schwarz inequality,
$$ 
\ba 
\left|\int_{\bR^3}\cQ(g_n)\phi(w)\,\dd w\right|&
\\
\le C\iint_{\bR^3\times\bR^3}M(w_*)g_n(w_*)M(w)g_n(w)(1+|w|+|w_*|)\,\dd w_*\dd w&
\\
\le C \int_{\bR^3}M(w)g_n(w)^2\,\dd w\int_{\bR^3}M(w)(1+|w|)^2\,\dd w&\,,
\ea
$$
so that
$$
\int_{\bR^3}\cQ(g_n)\phi(w)\,\dd w\hbox{ is bounded in }L^1_{loc}(\bR_+^*\times\bR^3)
$$ 
for each $\phi \in C_c(\bR^3)$, and $\eps_n \cQ(g_n)\to 0$ in the sense of distributions. 

Likewise, for each $\phi \in C_c(\bR^3)$, we deduce from (H1) that
$$
\ba 
\int_{\bR^3}\cR(f_n,F_n)M^{-1}(w)\phi(w)\,\dd w&
\\
=\iiint_{\bR^3\times\bR^3\times\bR^3}(M^{-1}(w)\phi(w)-M^{-1}(W)\phi(W))f_n(W)F_n(V)\Pi_{gp}(w,\dd V\dd W)\,\dd w&\,,
\ea
$$
so that 
$$
\ba
\left|\int_{\bR^3}\cR(f_n,F_n)M^{-1}(w)\phi(w)\,\dd w\right|
\\
\le C\iint_{\bR^3\times\bR^3}F_n(V)f_n(W)q(|\eps_n V - W|)\,\dd V\dd W
\\
\le C\int_{\bR^3}M(W)(1+\eps_ng_n)(W)(1+|W|)\,\dd W\,,  
\ea
$$
which is bounded in $L^2_{loc}(\bR_+^* \times \bR^3)$, according to (H1) and assumption (b) in Theorem \ref{T-LimThm}. Therefore $\eps_n\cR(f_n,F_n)M^{-1}(w)\to 0$ in the sense of distributions.

Finally, for each test function $\phi\in C_c(\bR^3)$, one has\footnote{We use the notation
$$
(\phi|\psi)_{L^2(M,\dd v)}:=\int_{\bR^3}\overline{\phi(v)}\psi(v)M(v)\,\dd v\,,\quad\hbox{ for each }\phi,\psi\in L^2(M\,\dd v)\,.
$$}: 
$$
(\cL g_n|\phi)_{L^2(M\,\dd v)}=(g_n|\cL\phi)_{L^2(M\,\dd v)}\to(g|\cL\phi)_{L^2(M\,\dd v)}=(\cL g|\phi)_{L^2(M\,\dd v)}
$$
since $\cL\phi\in L^2(M\,\dd v)$ and $g_n\wto g$ in $L^2_{loc}(\bR_+^*\times\bR^3\times\bR^3)$ weak. Hence
$$
\cL g_n\to\cL g=0\hbox{ in the sense of distributions.}
$$
According to (\ref{null}), $g$ is of the form (\ref{eqr}).
\end{proof}

\subsubsection{Step 2: Asymptotic deflection term.}


The following proposition is the key observation in this work. Because the mass ratio of the gas molecules to the particles in the dispersed phase is assumed to be small, the heavier particles are only slightly deflected upon colliding 
with the lighter gas molecules. It explains how the collision integral $\cD(F,f)$ in the kinetic equation for the distribution function of the dispersed phase converges to the acceleration term which appears in the Vlasov equation. This
result is reminiscent of Theorem 4.3 in \cite{DegoLucq}.

\begin{proposition}\label{defl}
Under the assumptions of Theorem \ref{T-LimThm},
$$
\frac1{\eta_n}\cD(F_n,f_n)\to\ka\Div((v-u)F)\quad\hbox{ in }\cD'(\bR_+^*\times\bR^3\times\bR^3)\,,
$$
with $\ka$ defined in (\ref{nuka}). More precisely, for each $\phi\equiv\phi(v)\in C^2(\bR^3)$ such that $\grad\phi$ and $\grad^2\phi\in L^\infty(\bR^3)$, one has
$$
-\frac1{\eta_n}\int_{\bR^3}\cD(F_n,f_n)\phi(v)\,\dd v\to\ka\int_{\bR^3}F(v)\nabla\phi(v)\cdot(v-u)\,\dd v
$$ 
in $\cD'(\bR_+^*\times\bR^3)$.
\end{proposition}

\begin{proof}
Using (H1) or (\ref{Colli2}) and the Taylor expansion at order $2$ for the $C^2$ function $\phi$, one has
$$
\ba
\frac1\eta_n\int_{\bR^3}\cD(F_n,f_n)(v)\phi(v)\,\dd v&
\\
=\frac1\eta_n\iint_{\bR^3\times\bR^3}F_n(V)f_n(W)\int_{\bR^3}(\phi(v)-\phi(V))\Pi_{pg}(v,\dd V\dd W)\,\dd v&
\\
=\frac1\eta_n\iint_{\bR^3\times\bR^3}F_n(V)f_n(W)\grad\phi(V)\cdot\int(v-V)\Pi_{pg}(v,\dd V\dd W)\,\dd v&
\\
+\frac1\eta_n\iint_{\bR^3\times\bR^3}F_n(V)f_n(W)\int H(v,V):(v-V)^{\otimes 2}\Pi_{pg}(v,\dd V\dd W)\,\dd v&
\\
=: I_n+J_n&\,,
\ea
$$
where
$$
H(v,V):=\int_0^1(1-t)\grad^2\phi((1-t)V+tv)\,\dd t\,.
$$

\medskip
We first treat the term $I_n$. According to (H2)
$$
I_n=-\int_{\bR^3}F_n(V)\grad\phi(V)\cdot \frac{K_n(V)}{1+\eta_n}\,\dd V\,,
$$
where
$$ 
K_n(V):=\frac1{\eps_n}\int_{\bR^3}f_n(W)(\eps_nV-W)Q(|\eps_nV-W|)\,\dd W\,. 
$$
Hence
$$
I_n=I_n^1+I_n^2+I_n^3+I_n^4+I_n^5\,,
$$
with
$$
\ba
{}&I_n^1=-\eps_n\int_{\bR^3}F_n(V)\frac{\grad\phi(V)}{1+\eta_n}\cdot\int_{\bR^3}M(W)g_n(W)VQ(|\eps_n V-W|)\,\dd W\dd V\,,
\\
&I_n^2=\int_{\bR^3}F_n(V)\frac{\grad\phi(V)}{1+\eta_n}\cdot\int_{\bR^3}M(W)g_n(W)W(Q(|\eps_n V-W|) - Q(|W|)\,\dd W\,,
\\
&I_n^3=\int_{\bR^3}F_n(V)\frac{\grad\phi(V)}{1+\eta_n}\cdot\int_{\bR^3}M(W)g_n(W)WQ(|W|)\,\dd W\dd V\,,
\\
&I_n^4=-\int_{\bR^3}F_n(V)\frac{\grad\phi(V)}{1+\eta_n}\cdot\int_{\bR^3}M(W)VQ(|\eps_n V-W|)\,\dd W\dd V\,,
\\
&I_n^5=\frac1{\eps_n}\int F_n(V)\frac{\grad\phi(V)}{1+\eta_n}\cdot\int M(W)W(Q(|\eps_n V-W|)-Q(|W|))\,\dd W\dd V\,.
\ea
$$

Notice that
$$
\ba
\left|\int_{\bR^3}M(W)g_n(W)VQ(|\eps_n V-W|)\,\dd W\right| 
\\
\le C\int_{\bR^3}(1+\eps_n|V|+|W|)M(W)|g_n(W)||V|\,\dd W 
\\
\le C|V|(1+|V|)\sqrt{\int_{\bR^3}M g_n^2\,\dd W}
\ea
$$ 
by the Cauchy-Schwarz inequality, so that
$$ 
I_n^1\to 0\quad\hbox{ in }L^2_{loc}(\bR_*^+ \times \bR^3)\,. 
$$
Then, 
$$
\ba
\left|\int_{\bR^3}M(W)g_n(W)W(Q(|\eps_n V-W|) - Q(|W|))\,\dd W\right| 
\\
\le C\eps_n \int_{\bR^3}M(W)|g_n(W)||W||V|(1+\eps_n|V|+|W|)\,\dd W
\\
\le C\eps_n(1+|V|^2)\sqrt{\int_{\bR^3}M(W)g_n^2\,\dd W}\,,
\ea
$$ 
so that
$$
I_n^2\to 0\quad\hbox{ in }L^2_{loc}(\bR_*^+ \times \bR^3)\,. 
$$
By assumption (d) in Theorem \ref{T-LimThm}
$$
\ba
\int_{\bR^3}M(W)g_n(W)WQ(|W|)\,\dd W\to&\int_{\bR^3}M(W)g(W)WQ(|W|)\,\dd W
\\
&=\tfrac13u\int_{\bR^3}M(W)|W|^2Q(|W|)\,\dd W=\ka u
\ea
$$
in $L^1_{loc}(\bR_*^+\times \bR^3)$, and therefore
$$ 
I_n^3\to\ka u\cdot\int_{\bR^3}F(V)\grad\phi(V)\,\dd V\quad\hbox{ in }\cD'(\bR_*^+ \times \bR^3)\,.
$$
Then
$$
\ba
\left|\int_{\bR^3}M(W)V(Q(|\eps_n V-W|) - Q(|W|))\,\dd W\right|&
\\
\le C \eps_n\int_{\bR^3}M(W)|V|^2(1+\eps_n|V|+|W|)\,\dd W&
\\
\le C\eps_n|V|^2(1+|V|)&\,,
\ea
$$
so that
$$
\int_{\bR^3}F_n(V)\frac{\grad\phi(V)}{1+\eta_n}\cdot\int_{\bR^3}M(W)V(Q(|\eps_n V-W|) - Q(|W|))\,\dd W\dd V\to  0
$$
locally uniformly on $\bR_*^+\times\bR^3$, and
$$
I_n^4\to-\int_{\bR^3}F(V)\grad\phi(V)\cdot V\int_{\bR^3}M(W)Q(|W|)\,\dd W\dd V\hbox{ in }\cD'(\bR_*^+ \times \bR^3)\,.
$$ 
Finally 
$$ 
\ba
\left|\int_{\bR^3}M(W)W\left(\frac{Q(|\eps_n V-W|)-Q(|W|)}{\eps_n}+\frac{W}{|W|}\cdot VQ'(|W|)\right)\,\dd W\right| 
\\
\le\int_{\bR^3}M(W)|W||V|\int_0^1|Q'(|\theta \eps_n V-W|)-Q'(|W|)|\,\dd\theta\dd W
\\
\le C|V|(1+|V|)
\ea
$$
and
$$
\int_{\bR^3}M(W)|W||V|\int_0^1|Q'(|\theta \eps_n V-W|)-Q'(|W|)|\,\dd\theta\dd W\to 0
$$
for all $V\in\bR^3$ by dominated convergence. With assumption (b) in Theorem \ref{T-LimThm}, we see that
$$
I_n^5+\int_{\bR^3}F_n(V)\frac{\grad\phi(V)}{1+\eta_n}\cdot\int_{\bR^3}M(W)\frac{W}{|W|} W\cdot V Q'(|W|)\,\dd W\dd V\to 0
$$
locally uniformly on $\bR_*^+ \times \bR^3$, and therefore
$$
I_n^5\to-\int_{\bR^3}F(V)\grad\phi(V)\cdot\int_{\bR^3}M(W)\frac{W}{|W|} W\cdot V Q'(|W|)\,\dd W\dd V\hbox{ in }\cD'(\bR_*^+ \times \bR^3)\,.
$$
By isotropy, one has
$$
\int_{\bR^3}M(W)\frac{W}{|W|} W\cdot V Q'(|W|)\,\dd W=\tfrac13V\int_{\bR^3}M(W)|W|Q'(|W|)\,\dd W\,,
$$
so that
$$
I_n^4+I_n^5\to-\int_{\bR^3}F(V)\grad\phi(V)\cdot V\int_{\bR^3}M(W)(Q(|W|)+\tfrac13|W|Q'(|W|))\,\dd W\dd V
$$
in $\cD'(\bR_*^+ \times \bR^3)$. On the other hand, we observe that
$$
W\cdot\grad M(W)=-|W|^2M(W)
$$
so that
$$
\ba
\int_{\bR^3}M(W)|W|^2Q(|W|)\,\dd W=&-\int_{\bR^3}W\cdot\grad M(W)Q(|W|)\,\dd W
\\
=&\int_{\bR^3}M(W)\Div(WQ(|W|))\,\dd W
\\
=&\int_{\bR^3}M(W)(3Q+|W|Q')(|W|)\,\dd W\,.
\ea
$$
Hence
$$
\ba
I_n^4+I_n^5\to&-\int_{\bR^3}F(V)\grad\phi(V)\cdot V\int_{\bR^3}\tfrac13|W|^2M(W)Q(|W|)\,\dd W\dd V
\\
=&-\ka\int_{\bR^3}F(V)\grad\phi(V)\cdot V\,\dd V
\ea
$$
in $\cD'(\bR_*^+ \times \bR^3)$. Therefore
$$
I_n\to-\ka\int_{\bR^3}F(V)\grad\phi(V)\cdot(V-u)\,\dd V\quad\hbox{ in }\cD'(\bR_*^+ \times \bR^3)\,,
$$
with
$$
\ka=\tfrac13\int_{\bR^3}M(W)Q(|W|)|W|^2\,\dd W>0\,.
$$

\medskip
Next we treat the term $J_n$. One has
$$
|J_n|\le\frac1{2\eta_n}\|\grad^2\phi\|_{L^\infty}\iint_{\bR^3\times\bR^3}F_n(V)f_n(W)\int_{\bR^3}|v-V|^2\Pi_{pg}(v,\dd V\dd W)\,\dd v\,.
$$
With $U=\frac{\eps_nV+\eta_n W}{1+\eta_n}$, one has 
$$
\ba
\eps_n^2|v-V|^2&\le 2|\eps_nv-U|^2+2|U-\eps_nV|^2
\\
&=2|\eps_nv-U|^2+\frac{2\eta_n^2}{(1+\eta_n)^2}|\eps_nV-W|^2\,.
\ea
$$
According to assumption (H3),
$$
\int_{\bR^3}|v-V|^2\Pi_{pg}(v,\dd V\dd W)\,\dd v\le\frac{2C}{\eps^2_n}\eta_n^2(1+|\eps_nV-W|^2)q(|\eps_nV-W|)\,,
$$
so that
$$
|J_n|\le\frac{C}{\eps_n^2}\eta_n||\nabla^2 \phi||_{L^{\infty}}\iint_{\bR^3\times\bR^3}F_n(V)f_n(W)(1+|\eps_nV-W|^2)q(|\eps_nV-W|)\,\dd V\dd W\,.
$$
By (H1) and assumption (b) in Theorem \ref{T-LimThm}, 
$$
\ba
\iint_{\bR^3\times\bR^3}F_n(V)f_n(W)(1+|\eps_nV-W|)^2q(|\eps_nV-W|)\,\dd V\dd W
\\
\le CC_R\iint_{\bR^3\times\bR^3}\frac{(1+\eps_n|V|+|W|)^3}{(1+|V|)^p}M(W)(1+\eps_ng_n)(W)\,\dd V\dd W
\ea
$$
for $(t,x)\in[0,R]\times[-R,R]^3$. By assumption (c) in Theorem \ref{T-LimThm} and the Cauchy-Schwarz inequality, the right hand side is bounded in $L^2_{loc}(\bR_*^+\times\bR^3)$. Hence
$$
J_n\to 0\quad\hbox{ in }L^2_{loc}(\bR_*^+\times\bR^3)
$$
since $\eta_n/\eps_n^2\to 0$, which concludes the proof of Proposition \ref{defl}.
\end{proof}

\subsubsection{Step 3: Asymptotic friction term.}


\begin{proposition}\label{fric}
Under the assumptions of Theorem \ref{T-LimThm},
$$
\frac1{\eps_n}\int_{\bR^3}w\cR(f_n,F_n)\,\dd w\to\ka\int_{\bR^3}(v-u)F\,\dd v\quad\hbox{ in }\cD'(\bR_+^*\times\bR^3),
$$
with $\ka$ defined by formula (\ref{nuka}). 
\end{proposition}

\begin{proof}
By assumptions (H1)-(H2),
$$
\ba
\frac1{\eps_n}\int_{\bR^3}w\cR(f_n,F_n)\,\dd w&
\\
=\frac1{\eps_n}\iiint_{\bR^3\times\bR^3\times\bR^3}(w-W)f_n(W)F_n(V)\Pi_{gp}(w,\dd V\dd W)\,\dd w&
\\
=-\frac1{\eta_n}\iiint_{\bR^3\times\bR^3\times\bR^3}(v-V)F_n(V)f_n(W)\Pi_{pg}(v,\dd V\dd W)\,\dd v&
\\
=-\frac1{\eta_n}\int_{\bR^3}v\cD(F_n,f_n)\,\dd v&\,.
\ea
$$
Proposition \ref{defl} then implies that
\be\lb{WeakLimDefl}
-\frac1{\eta_n}\int_{\bR^3}\phi(v)\cD(F_n,f_n)\,\dd v\to\ka\int_{\bR^3}F(V)\grad\phi(V)\cdot (V-u)\,\dd V
\ee
in $\cD'(\bR_+^*\times\bR^3)$ for each test function $\phi\equiv\phi(v)$ satisfying
$$
\phi\in C^2(\bR^3)\quad\hbox{ and }\grad\phi,\grad^2\phi\in L^\infty(\bR^3)\,.
$$
Setting $\phi(v)=v$ in (\ref{WeakLimDefl}) leads to the conclusion.
\end{proof}

\subsubsection{Step 4: Incompressibility condition.}


\begin{proposition}\label{bouss}
Under the assumptions of Theorem \ref{T-LimThm}, the velocity field $u$ satisfies the incompressibility condition
\be\label{incomp}
 \Div_xu=0
\ee
in the sense of distributions on $\bR_+^*\times\bR^3$.
\end{proposition}

\begin{proof}
For each $\phi:=\phi(w)\in L^1(M\,\dd v)$, we set
\be\label{deflara}
\la\phi\ra:=\int_{\bR^3}\phi(w)M(w)\,\dd w\,. 
\ee
Multiplying both sides of (\ref{eqg}) by $\eps_nM(w)$ and integrating in $w$ shows that
$$
\eps_n\d_t\la g_n\ra+\Div_x\la wg_n\ra=0\,.
$$
according to (\ref{star}). Since $g_n\wto g$ in $L^2(\bR_+^*\times\bR^3\times\bR^3)$ weak and satisfies assumption (c) in Theorem \ref{T-LimThm}, 
$$
\la g_n\ra\to\la g\ra\quad\hbox{ and }\la wg_n\ra\to\la wg\ra\hbox{ in }L^2_{loc}(\bR_+^*\times\bR^3)\hbox{ weak.}
$$
Hence
$$
\Div_x\la wg_n\ra=-\eps_n\d_t\la g_n\ra\to 0\hbox{ in }\cD'(\bR_+^*\times\bR^3),
$$
so that
$$
\Div_x\la wg\ra=0\,.
$$
According to Proposition \ref{rhout},  one has $\la wg\ra=u$, so that (\ref{incomp}) holds.
\end{proof}

\subsubsection{Step 5: Viscosity term.}


\begin{proposition}\label{visc}
Under the assumptions of Theorem \ref{T-LimThm},
$$
\la\tilde A(w)w\cdot\grad_xg\ra=\nu(\grad_xu+(\grad_xu)^T)\,,
$$
where $\tilde A$ is defined in (\ref{defAtilde}), and $\nu$ is defined in (\ref{nuka}). 
\end{proposition}

\begin{proof}
By Proposition \ref{rhout}, one has
$$
\la\tilde A(w)w\cdot\grad_xg\ra=\la\tilde A(w)\otimes A(w)\ra:\grad_xu
$$
since the tensor field $w\mapsto A(w)w$ is odd. By Lemma 4.4 in \cite{BGL2} (see formula (4.13a)), one has
$$
\la\tilde A_{ij}A_{kl}\ra=\nu(\de_{ik}\de_{jl}+\de_{il}\de_{jk}-\tfrac13\de_{ij}\de_{kl}),
$$
with 
$$
\nu:=\tfrac1{10}\la\tilde A:\cL\tilde A\ra>0
$$
(see formula (4.10) in \cite{BGL2}). Formulas (4.10)-(4.13a) in \cite{BGL2} are based on elementary symmetry arguments --- most notably the fact that $A(Rw)=RA(w)R^T$ for each $R\in O_3(\bR)$. Complete proofs of these formulas can be found 
in Lemma 4.3 of \cite{GoB}. Hence
$$
\la\tilde A(w)w\cdot\grad_xg\ra=\nu(\grad_xu+(\grad_xu)^T-\tfrac23(\Div_xu)I)\,.
$$
Since the velocity field $u$ is divergence-free by Proposition \ref{bouss}, this concludes the proof.
\end{proof}

\subsubsection{Step 6: Convection term.}


\begin{proposition}\label{conv}
Under the assumptions of Theorem \ref{T-LimThm},
$$ 
\la\tilde A(w)\cQ(g)\ra=A(u)
$$
where $\tilde A$ is defined in (\ref{defAtilde}), while $\cQ$ is defined in (\ref{defQ}).
\end{proposition}

\begin{proof}
By Proposition \ref{rhout}, $g(t,x,\cdot)\in \Ker\cL$ for a.e. $(t,x)\in\bR_+^*\times\bR^3$. According to formula (60) in \cite{BGL1}, one has
$$
\cQ(g(t,x,\cdot))=\tfrac12\cL(g(t,x,\cdot)^2)\,,\quad\hbox{ for a.e. }(t,x)\in\bR_+^*\times\bR^3\,.
$$
Since $\cL$ is self-adjoint on $L^2(M\,\dd w)$ by Theorem \ref{T-HilbGr} and $g^2\in\Dom\cL$, one has
$$
\la\tilde A(w)\cQ(g)\ra=\la\tilde A(w)\tfrac12\cL(g^2)\ra=\tfrac12\la(\cL\tilde A)g^2\ra=\tfrac12\la Ag^2\ra\,.
$$
Eliminating the odd component of $g^2$ since $w\mapsto A(w)$ is even, one finds that
$$
\la Ag^2\ra=\la A\otimes w\otimes w\ra:(u\otimes u)+\La A\left(\rho+\th\tfrac12(|w|^2-3)\right)^2\Ra\,.
$$

First 
$$
\La A\left(\rho+\th\tfrac12(|w|^2-3)\right)^2\Ra=\tfrac13\La\Tr(A)\left(\rho+\th\tfrac12(|w|^2-3)\right)^2\Ra I=0
$$
because $A(Rw)=RA(w)A^T$ and $\Tr(A)=0$ --- see Lemma 4.2 in \cite{GoB} for a detailed proof. 

Then
$$
\la A\otimes w\otimes w\ra_{ijkl}=\la A_{ij}A_{kl}\ra=\de_{ik}\de_{jl}+\de_{il}\de_{jk}-\tfrac23\de_{ij}\de_{kl}\,,
$$
Lemma 4.2 in \cite{GoB}, so that
$$
\la A\otimes w\otimes w\ra:(u\otimes u)=2u\otimes u-\tfrac23|u|^2I\,.
$$
This concludes the proof.
\end{proof}

\medskip

\subsubsection{Step 7: Asymptotic friction flux.}


\begin{proposition}\label{last}
Under the assumptions of Theorem \ref{T-LimThm},
$$
\int\tilde A(w)\cR(f_n,F_n)(w)dw\to 0\quad\hbox{ in }\cD'(\bR_+^*\times\bR^3)\,.
$$
\end{proposition}

\begin{proof}
First, we deduce from (H1) that
$$ 
\ba
\int_{\bR^3}\tilde A(w)\cR(M,F_n)(w)\,\dd w&
\\
=\iiint_{\bR^3\times\bR^3\times\bR^3}F_n(V)M(W)(\tilde A(w)-\tilde A(W))\Pi^{\eps_n,\eta_n}_{gp}(w,\dd V\dd W)\,\dd w&\,.
\ea
$$
Then
$$ 
\ba
\left|\int_{\bR^3}\left(\tilde A\cR(M,F_n)\!-\!\iint_{\bR^3\times\bR^3}F(V)M(W)(\tilde A(w)\!-\!\tilde A(W))\Pi^{0,0}_{gp}(w,\dd V\dd W)\right)\,\dd w\right|
\\
\le\left|\iiint_{\bR^3\times\bR^3\times\bR^3}F_n(V)M(W)(\tilde A(w)-\tilde A(W))(\Pi^{\eps_n,\eta_n}_{gp}-\Pi^{0,0}_{gp})(w,\dd V\dd W)\,\dd w\right|&
\\
+\left|\iiint_{\bR^3\times\bR^3\times\bR^3}(F_n(V) - F(V))M(W)(\tilde A(w)-\tilde A(W))\Pi^{0,0}_{gp}(w,\dd V\dd W)\,\dd w\right|&\,.
\ea
$$
The first term on the right hand side vanishes in the sense of distributions on $\bR_+^*\times\bR^3$ because of the second part of assumption (H4) and the fact that the radial function $\a$ in (\ref{defalpha}) belongs to $L^\infty(\bR_+)$. 
The second term on the right hand side also vanishes in the sense of distributions on $\bR_+^*\times\bR^3$ because of the last part of assumption (H4).

According to the first part of assumption (H4),
$$
\ba
\iiint_{\bR^3\times\bR^3\times\bR^3}F(V)M(W)(\tilde A(w)-\tilde A(W))\Pi^{0,0}_{gp}(w,\dd V\dd W)\,\dd w&
\\
=\iiint_{\bR^3\times\bR^3\times\bR^3}F(V)M(W)(\tilde A(w)-\tilde A(W))\cT_R\#\Pi^{0,0}_{gp}(w,\dd V\dd W)\,\dd w&
\\
=\iiint_{\bR^3\times\bR^3\times\bR^3}F(V)M(W)(\tilde A(Rw)-\tilde A(RW))\Pi^{0,0}_{gp}(w,\dd V\dd W)\,\dd w&
\ea
$$
for each $R\in O_3(\bR)$, where $\cT_R$ is defined in (\ref{defTR}). Because of (\ref{defalpha}),
$$
\tilde A(Rw)=R \tilde A(w)R^T\,,\quad\hbox{ for each }R\in O_3(\bR).
$$
Thus, for each $R\in O_3(\bR)$,
$$
\cA:=\iiint_{\bR^3\times\bR^3\times\bR^3}F(V)M(W)(\tilde A(Rw)-\tilde A(RW))\Pi^{0,0}_{gp}(w,\dd V\dd W)\,\dd w=R\cA R^T
$$
a.e. on $\bR_+^*\times\bR^3$. At this point, we use the following classical lemma.

\begin{lemma} Let $\cM=\cM^T\in M_3(\bR)$ satisfy 
$$
R\cM=\cM R\hbox{ for each }R\in O_3(\bR)\,.
$$
Then $\cM$ is of the form 
$$
\cM=\l I\,,\quad\hbox{ with }\l=\tfrac13\Tr\cM\,.
$$
\end{lemma}

(The proof of this lemma is an easy exercise in linear algebra; alternately, it is a special case of Lemma 4.1 in \cite{GoB} for $m=2$ and in the case of a constant tensor field, i.e. $T(\xi)\equiv T(0)$.)

As a consequence,
$$
\cA(t,x)=\tfrac13\Tr(\cA(t,x))I=0\,,
$$
since
$$
\Tr\cA=\iiint_{\bR^3\times\bR^3\times\bR^3} F(V)M(W)\Tr(\tilde A(w)-\tilde A(W))\Pi^{0,0}_{gp}(dwdVdW)=0 .
$$
Hence
\be\lb{FricFlux1}
\int_{\bR^3}\tilde A(w)\cR(M,F_n)(w)\,\dd w\to 0\quad\hbox{ in }\cD(\bR_+^* \times \bR^3)\,.
\ee

Next, we deduce from (H1) that
\be\lb{FricFlux2}
\ba
\left|\int_{\bR^3}\cR(Mg_n, F_n)\tilde A(w)\,\dd w\right|
\\
=\left|\iiint_{\bR^3 \times\bR^3\times\bR^3}(\tilde A(w)-\tilde A(W))M(W)g_n(W)F_n(V)\Pi_{gp}(w,\dd V\dd W)\,\dd w\right|
\\
\le C_K\iiint_{\bR^3\times\bR^3\times\bR^3}(|w|^2+|W|^2)M(W)|g_n(W)|(1+|V|^2)^{-p}\Pi_{gp}(w,\dd V\dd W)\,\dd w
\\
\le CC_K\|g\|_{L^2(M\,\dd w)}
\ea
\ee
for all $(t,x)\in[0,K]\times[-K,K]^3$, by (H5) and assumptions (c) in Theorem \ref{T-LimThm}.

The conclusion follows from (\ref{FricFlux1})-(\ref{FricFlux2}), from assumption (c) in Theorem \ref{T-LimThm} showing the last right hand side of (\ref{FricFlux2}) is bounded in $L^2_{loc}(\bR_+^*\times\bR^3)$, and from the identity 
$$
\ba
\int_{\bR^3}\tilde A(w)\cR(f_n,F_n)(w)\,\dd w=&\int_{\bR^3}\tilde A(w)\cR(M,F_n)(w)\,\dd w
\\
&+\eps_n\int_{\bR^3}\tilde A(w)\cR(Mg_n,F_n)(w)\,\dd w\,.
\ea
$$
\end{proof}

\subsubsection{Step 8: End of the proof of Theorem \ref{T-LimThm}.}


First, we recall that $\cL$ is self-adjoint in $L^2(M\,\dd w)$ according to Theorem \ref{T-HilbGr}. Hence
$$
\frac1{\eps_n}\la A(w)g_n\ra=\frac1{\eps_n}\la(\cL\tilde A)(w)g_n\ra=\La\tilde A(w)\frac1{\eps_n}\cL g_n\Ra\,.
$$
Following the same procedure as in \cite{BGL1}, we use the Boltzmann equation for $g_n$ in the form (\ref{eqgpr}) to express the term $\frac1{\eps_n}\cL g_n$:
\be\lb{MomFlux}
\ba
\frac1{\eps_n}\la A(w)g_n\ra=&\la\tilde A(w)\cQ(g_n)\ra-\la\tilde A(w)(\eps_n\d_t+w\cdot\grad_x)g_n\ra
\\
&+\la\tilde A(w)M^{-1}\cR(f_n,F_n)\ra .
\ea
\ee

We first pass to the limit in the term $\la\tilde A(w)(\eps_n\d_t+w\cdot\grad_x)g_n\ra$ in the sense of distributions on $\bR_+^*\times\bR^3$. Since the function $\a\in(\ref{defalpha})$ is bounded, one has
$$
\int_{\bR^3}(1+|w|)^2|\tilde A(w)|^2M(w)\,\dd w<\infty\,.
$$
By assumption (c) in Theorem \ref{T-LimThm} and the Cauchy-Schwarz inequality, 
$$
\la\tilde Ag_n\ra\to\la\tilde Ag\ra\hbox{ and }\la w\tilde Ag_n\ra\to\la w\tilde Ag\ra\hbox{ in }L^2_{loc}(\bR_+^*\times\bR^3)\hbox{ weak.}
$$
Hence
$$
\la\tilde A(w)(\eps_n\d_t+w\cdot\grad_x)g_n\ra=\eps_n\d_t\la\tilde Ag_n\ra+\Div_x\la w\tilde Ag_n\ra\to\Div_x\la w\tilde Ag\ra
$$
in $\cD'(\bR_+^*\times\bR^3)$. By Proposition \ref{visc}, 
\be\lb{LimVisc}
\la\tilde A(w)(\eps_n\d_t+w\cdot\grad_x)g_n\ra\to\nu(\grad _xu+(\grad_xu)^T)\hbox{ in }\cD'(\bR_+^*\times\bR^3)\,.
\ee

Next we use the identity
$$ 
\la\tilde A\cQ(g_n)\ra=\iint_{\bR^3\times\bR^3}P(w,w_*)M(w_*)g_n(w_*)M(w)g_n(w)\dd w\dd w_*
$$
where
$$ 
P(w,w_*):=\int_{\bS^2}(\tilde A(w')-\tilde A(w))c(w-w_*,\om)\,\dd\om\,.
$$
Obviously
$$
\la\tilde A\cQ(g_n)\ra=\int_{\bR^3}h_n(t,x,w)M(w)g_n(w)\dd w
$$
with
$$
h_n(t,x,w):=\int_{\bR^3}P(w,w_*)M(w_*)g_n(t,x,w_*)\,\dd w_*\,.
$$

One has
$$
|P(w,w_*)|\le C(1+|w|^3+|w_*|^3)
$$
because of the growth assumption (\ref{cc4}) on the collision kernel, and the assumption that the function $\a$ in (\ref{defalpha}) is bounded on $\bR_+^*$. Assumption (c) in Theorem \ref{T-LimThm} implies that
$$
\sup_{n\ge 1}\iiint_{[0,R]\times[-R,R]^3\times\bR^3}M(w_*)g_n(t,x,w_*)^2\,\dd w_*\dd x\dd t<\infty
$$
so that, by the Cauchy-Schwarz inequality,
$$
\int_{|w_*|>R}|P(w,w_*)||g_n(t,x,w_*)|M(w_*)\,\dd w_*\to 0\hbox{ in }L^2_{loc}(\bR_+^*\times\bR^3\times\bR^3)
$$
uniformly in $n\ge 1$ as $R\to\infty$. Therefore, we deduce from assumption (d) in Theorem \ref{T-LimThm} that
$$
\ba
h_n(t,x,w)\to\int_{\bR^3}P(w,w_*)M(w_*)g(t,x,w_*)\,\dd w_*=:h(t,x,w)
\ea
$$
in $L^2_{loc}(\bR_+\times\bR^3\times\bR^3)$. In particular, by weak-strong continuity of the pointwise product, one has
$$
\int_{|w|\le K}h_n(t,x,w)M(w)g_n(t,x,w)\dd w\to\int_{|w|\le K}h(t,x,w)M(w)g(t,x,w)\dd w
$$
in $\cD'(\bR_+\times\bR^3)$ for all $K>0$. On the other hand
$$
M(w)h_n(t,x,w)^2\le C(1+|w|^3)^2M(w)\int_{\bR^3}M(w_*)g_n(t,x,w_*)^2\,\dd w_*\,,
$$
so that, by the Cauchy-Schwarz inequality,
$$
\ba
\int_{|w|>K}h_n(t,x,w)M(w)g_n(t,x,w)\dd w
\\
\le\sqrt{C}\left(\int_{|w|>K}(1+|w|^3)^2M(w)\,\dd w\right)^{1/2}\int_{\bR^3}M(\xi)g_n(t,x,\xi)^2\,\dd \xi\to 0
\ea
$$
in $L^2_{loc}(\bR_+^*\times\bR^3)$ as $K\to+\infty$ uniformly in $n\ge 1$, according to assumption (c) in Theorem \ref{T-LimThm}. Hence
\be\lb{LimConv}
\ba
\la\tilde A\cQ(g_n)\ra(t,x)=&\int_{\bR^3}h_n(t,x,w)M(w)g_n(t,x,w)\dd w
\\
&\to\int_{\bR^3}h(t,x,w)M(w)g(t,x,w)\dd w=\la\tilde A\cQ(g)\ra(t,x)=A(u)(t,x)
\ea
\ee
in $\cD'(\bR_+\times\bR^3)$, where the last equality follows from Proposition \ref{conv}.

Since the last term on the right hand side of (\ref{MomFlux}) vanishes by Proposition \ref{last}, we conclude that
$$
\frac1{\eps_n}\la A(w)g_n\ra\to A(u)-\nu\left((\grad_xu)+(\grad_xu)^T\right)\quad\hbox{ in }\cD'(\bR_+\times\bR^3)\,.
$$
In particular,
$$
\ba
\Div_x\frac1{\eps_n}\la A(w)g_n\ra\to\Div_x(u\otimes u)-\tfrac13\grad_x|u|^2-\nu\Dlt_xu-\nu\grad_x\Div_xu
\\
=\Div_x(u\otimes u)-\nu\Dlt_xu-\tfrac13\grad_x|u|^2
\ea
$$
in $\cD'(\bR_+^*\times\bR^3)$, by the divergence-free condition in Proposition \ref{bouss}. Hence, for each divergence-free, compactly supported, smooth vector field $\xi\equiv\xi(x)\in\bR^3$,
$$
\ba
\int_{\bR^3}\frac1{\eps_n}\la w\otimes wg_n\ra(t,x):\grad\xi(x)\,\dd x=\int_{\bR^3}\frac1{\eps_n}\la A(w)g_n\ra(t,x):\grad\xi(x)\,\dd x
\\
\to\int_{\bR^3}(u\otimes u-\nu\grad_xu)(t,x):\grad\xi(x)\,\dd x
\ea
$$ 
in $\cD'(\bR_+^*)$.
\medskip

We recall that the momentum balance law for the Boltzmann equation for gas molecules is
\be\label{newlin}
\d_t\la wg_n\ra+\frac1{\eps_n}\Div_x\la w^{\otimes 2}g_n\ra=\frac1{\eps_n}\la wM^{-1}\cR(f_n,F_n)\ra\,.
\ee
By Proposition \ref{rhout}, 
$$
\la wg_n\ra\to\la wg\ra=u\quad\hbox{ in }L^2(\bR_+^*\times\bR^3)\hbox{ weak,}
$$
while
$$
\frac1{\eps_n}\la wM^{-1}\cR(f_n,F_n)\ra\to\ka\int(v-u)Fdv\quad\hbox{ in }\cD'(\bR_+^*\times\bR^3)\,.
$$

Thus, for each divergence-free, compactly supported, smooth vector field $\xi\equiv\xi(x)\in\bR^3$, passing to the limit in the weak formulation (in $x$) of the momentum balance law (\ref{newlin}), i.e.
$$
\ba
\d_t\int_{\bR^3}\xi(x)\cdot\la wg_n\ra(t,x)\,\dd x-&\frac1{\eps_n}\int_{\bR^3}\la A(w)g_n\ra(t,x):\grad\xi(x)\,\dd x
\\
&=\frac1{\eps_n}\int_{\bR^3}\xi(x)\cdot\la wM^{-1}\cR(f_n,F_n)\ra(t,x)\,\dd x\,,
\ea
$$
results in the equality
$$
\ba
\d_t\int_{\bR^3}u(t,x)\cdot\xi(x)\,\dd x=&\int_{\bR^3}(u\otimes u-\nu\grad_xu)(t,x):\grad\xi(x)\,\dd x
\\
&+\ka\iint_{\bR^3\times\bR^3}\xi(x)\cdot(v-u(t,x))F(t,x,v)\,\dd v\dd x\,.
\ea
$$
By de Rham's characterization of currents homologous to $0$ (see Thm. 17' in \cite{deRham}), there exists $p\in\cD'(\bR_+^*\!\times\!\bR^3)$ such that
$$
\d_tu+\Div_x(u\otimes u-\nu\grad_xu)-\ka\int_{\bR^3}(v-u)F\,\dd v=-\grad_xp\,.
$$

Finally, we recall the equation for the distribution function of the dispersed phase:
$$
\d_tF_n+v\cdot\grad_xF_n=\frac1{\eta_n}\cD(F_n,f_n)\,.
$$
The assumptions on the convergence of $F_n$ in Theorem \ref{T-LimThm} imply that
$$
\d_tF_n+v\cdot\grad_xF_n\to\d_tF+v\cdot\grad_xF\quad\hbox{ in }\cD'(\bR_+^*\times\bR^3\times\bR^3)\,.
$$
Applying Proposition \ref{defl} shows that
$$
\d_tF+v\cdot\grad_xF=\ka\Div_v((v-u)F).
$$
and this concludes the proof of Theorem \ref{T-LimThm}.


\section{Conclusions and perspectives}


We conclude this paper with a few remarks on the method presented here, and on the assumptions used in Theorem \ref{T-LimThm}.

Let us first discuss the class of collision interactions considered in this work. 

We have assumed that intermolecular collisions correspond to cut-off hard potentials, which is natural. However, our assumption that the radial function $\alpha$ in (\ref{defalpha}) is bounded could be a significant restriction to the class
of intermolecular potentials considered. At present, this assumption is known to be satisfied only in the case of cut-off Maxwell molecules, when $\a$ is a constant. It would be natural to expect that the growth of $\a$ at infinity is such that
$$
\a(|w|)\sim(\bar c\star M(w))^{-1}\quad\hbox{ as }|w|\to\infty\,,
$$
however, we are not aware of any result  of this type in the existing literature, and we have not been able to prove it, even in the simplest case of hard sphere collisions. Perhaps the assumption that $\a$ is bounded can be relaxed at the
expense of more technical proofs.

Likewise, we have considered in this paper only the case of a monatomic propellant; however, this assumption could certainly be relaxed and more realistic models of propellant could be handled with the same methods.

Concerning collisions between gas molecules and dust particles/droplets, the hard spheres model for the collision cross-section may be the best choice when the detail of the interaction is not known, because the dust particles/droplets, 
though tiny, are macroscopic objects if compared to gas molecules. Hard spheres clearly belong to the class of cross-sections included in the assumptions of our theorem.

Otherwise, it would be more realistic to include polydispersion in our model of aerosol/spray --- i.e. to assume that the dust particles/droplets are distributed in size, and to include aggregation and fragmentation effects in the equation
for the distribution function of the dispersed phase. Such a generalization of the model considered in this work would be extremely natural, although significantly more technical, and we have avoided these effects in the present paper
for the sake of simplicity.

Finally, a few remarks on the class of solutions considered in this work are in order.

We have not tried to optimize the assumptions bearing on the solutions to the coupled Boltzmann system. Assuming a uniform control $g^2_n M$ is fairly natural, since quantity appears naturally in the entropy estimate. See section 3
in \cite{BGL2}, and Proposition 2.3 in \cite{GSR2} for a detailed discussion of this point.

As for $F_n$, we have chosen a $L^\infty$ setting since, in the limit, $F$ will satisfy a Vlasov equation which propagates $L^\infty$ estimates over finite time intervals. Finally, only the averages with respect to $w$ of $g_n$ are required 
to converge strongly in $L^2_{loc}$ and a.e., as in \cite{BGL1}. No such assumption is required on the averages of $F_n$ since no quadratic term in $F$ appears in the limit. 

Finally, the mathematical value of a formal limit theorem such as Theorem \ref{T-LimThm} can be questioned. However, it can be argued that the moment approach in the formal limit theorem \cite{BGL1} is the basis of the rigorous proofs
of the hydrodynamic limit of the Boltzmann equation leading to the incompressible Navier-Stokes equation in \cite{GSR1,GSR2}. Perhaps the hydrodynamic limit for the propellant is the most difficult part in Theorem \ref{T-LimThm}, 
and a rigorous derivation of the formal limit discussed here can be obtained along the lines of \cite{GSR1,GSR2}. That the discussion specific to the interaction of the propellant with the dispersed phase, i.e.  steps 2,3 and 7 in the proof of 
Theorem \ref{T-LimThm}, can be isolated from the Navier-Stokes limit for the propellant suggests that this derivation could be made rigorous in the not too distant future.


{\bf{Acknowledgment}}: The research leading to this paper was funded by
the French ``ANR blanche'' project Kibord: ANR-13-BS01-0004, and by Universit\'e Sorbonne Paris Cit\'e, in the framework of the ``Investissements d'Avenir'', convention ANR-11-IDEX-0005. V.Ricci acknowledges the support by the 
GNFM (research project 2015: ``Studio asintotico rigoroso di sistemi a una o pi\`u componenti'').




\end{document}